\documentclass[11pt]{article}
\usepackage[normalem]{ulem}
\newcommand{\mathsout}[1]

\usepackage{setspace}

\usepackage[margin=1in]{caption}
\usepackage{graphicx}
\usepackage{subfigure}
\usepackage{float}
\usepackage{enumitem}
\usepackage[bottom]{footmisc}
\usepackage{appendix}
\graphicspath{ {images/} }
\usepackage{amsmath,amsthm,amssymb}
\usepackage{mathrsfs}
\numberwithin{equation}{section}
\usepackage{subdepth}
\usepackage{bm}

\usepackage{cite}
\usepackage[colorlinks=true]{hyperref}

	\usepackage{environ}
	\usepackage{xparse}
	\usepackage[margin=1in]{geometry}
	\usepackage[numbers,sort]{natbib}
	\theoremstyle{definition}
	\newtheorem{definition}{Definition}[section]
    
    \newtheorem{remark}{Remark}[section]
	\theoremstyle{theorem}
    \newtheorem{corollary}{Corollary}[section]
	\newtheorem{theorem}{Theorem}[section]
    \newtheorem{lemma}{Lemma}[section]
       
	\theoremstyle{remark}
	
    \usepackage{thmtools}
    \usepackage[capitalise]{cleveref}

\newcommand\numberthis{\addtocounter{equation}{1}\tag{\theequation}}
\newcommand{\R}{\ensuremath{\mathbb{R}}}

\NewDocumentCommand{\integ}{m m O{} D<>{}}{
\IfNoValueTF{#3}{\ensuremath{\displaystyle \int_{#2} #4 \, \mathrm{d}#1}}
{\ensuremath{\displaystyle \int_{#2}^{#3} #4 \, \mathrm{d}#1}}
}

\newcommand{\norm}[1]{\left\lVert#1\right\rVert}
\newcommand{\e}[1]{\ensuremath{\exp{\left(#1\right)}}}
\let\RealFrac\frac
\NewDocumentCommand\f{mg}{%
  \IfNoValueTF{#2}
    {\RealFrac{1}{#1}}
    {\RealFrac{#1}{#2}}%
}

\newcommand{\ol}{\overline}

	\title{
		{Existence and Uniqueness of Traveling Fronts in Lateral Inhibition Neural Fields with Sigmoidal Firing Rates}\\
	}
	\author{Alan Dyson\footnote{Department of Mathematical Sciences, Lycoming College, 700 College Place, Williamsport, PA 17701-5192 (dysona33@gmail.com)}}
	\date{\today}
\begin{document}
\maketitle
\abstract 
We rigorously prove the existence of traveling fronts in neural field models with lateral inhibition coupling types and smooth sigmoidal firing rates. With Heaviside firing rates as our base point (where unique traveling fronts exist), we repeatedly apply the implicit function theorem in Banach spaces to provide a non-monotone version of the homotopy approach originally proposed by Ermentrout and McLeod (1993) in their seminal study of monotone fronts in purely excitatory models. By comparing smooth and Heaviside firing rates, we develop global wave speed and profile comparisons that guide our analysis, leading to uniqueness (modulo translation) in the perturbative case. Moreover, we establish a meaningful {\it a priori} existence result;
we prove existence holds for a range of firing rates, independent of continuation path.

\section{Introduction}\label{Section: Intro}
Traveling waves of voltage propagation are novel, network-level, neurophysiological patterns which researchers are studying across multiple disciplines. Such propagations can be readily observed in experimentation, which breaks down into {\it in vivo} and {\it in vitro} types. The combination of advanced electrode recording technology, voltage-sensitive dyes \cite{ferezou2006visualizing,petersen2003spatiotemporal}, and methods to pharmacologically mediate inhibition (such as delivering bicuculline, a $\textnormal{GABA}_{\textnormal{A}}$ antagonist used in numerous applications \cite{leventhal2003gaba,Benucci2007}) allows such patterns to be seen. Some examples of traveling waves include in the mammalian visual cortex \cite{Sato2012,Lee2005,nauhaus2009stimulus,Benucci2007} (and during binocular rivalry \cite{wilson2001dynamics}), primary somatosensory cortexes of Wistar rats \cite{Golomb2011} and rodents \cite{PintoandErmentrout-SpatiallyStructuredActivityinSynapticallyCoupledI.TravelingWaves,Connors-Generationofepileptiform}, and human \cite{zhang2015traveling} and guinea-pig \cite{traub1993analysis} hippocampus. Traveling waves are also known to be features of pathological disorders of the neocortex. For example, they have been observed during epileptiform discharges \cite{Connors-Generationofepileptiform,traub1993analysis,Golomb2011,wagner2015microscale} and migraines \cite{Lance-CurrentConceptsofmigraine}. With biologically motivated examples in mind, we study traveling waves in the context of one-dimensional, continuum-based models. Such an approach is justified based on the layered structure of the cortex, and has been widely pursued in mathematical neuroscience.
\subsection{Main Goal}
In this paper, we solve a particular subset of the traveling wave problem that is biologically motivated, but unsolved rigorously in continuum-based models---existence of traveling fronts in neural field models with smooth Heaviside (sigmoidal) firing rates and lateral inhibition coupling types. Such coupling types are widely assumed to represent activity {\it in vivo}, particularly in the visual cortex (see \cite{Sato2012} and references within), where traveling waves exist.

Specifically, with $\theta>0$ fixed, we are interested in the existence and uniqueness of traveling front solutions, $u(x,t)=U_{\tau}(x+\mu_{\tau} t),$ at wave speed $\mu_{\tau}>0,$ to the integro-differential equation \cite{PintoandErmentrout-SpatiallyStructuredActivityinSynapticallyCoupledI.TravelingWaves,ExistenceandUniqueness-ErmMcLeod,Amari1977}
\begin{equation} \label{eq:chSig_main}
u_t + u = \integ{y}{\mathbb{R}}[]<K(x-y)S_{\theta,\tau}(u(y,t))>.
\end{equation}
Under the traveling wave ansatz, the solution solves $F[U_{\tau},\mu_{\tau},S_{\theta,\tau}](z)=0,$ where
\begin{equation} \label{eq: mapping}
 F[U,\mu,S_{\theta,\tau}](z):= \mu U'(z) + U(z) -\integ{y}{\mathbb{R}}[]<K(z-y)S_{\theta,\tau}(U(y))>.
\end{equation}
Here, $u=u(x,t)$ represents average voltage in the spatial patch at position $x$ and time $t$; $K$ is a lateral inhibition (also called Mexican hat) synaptic coupling kernel in the sense that there exists $M>0$ such that $K(\cdot)> 0$ on $(-M,M)$ and $K(\cdot)<0$ on $(-\infty,-M)\cup (M,\infty)$; $\theta\in\left(0,\f{2}\right)$ is the minimum voltage threshold for coupling interactions. The monotone increasing, $C^\infty$ function, $S_{\theta,\tau}(u)\in[0,1],$ is the voltage-dependent firing rate for the network, which describes the proportion of presynaptic patches that are interacting with postsynaptic patches. For $\tau>0,$ we refer to $S_{\theta,\tau}$ as a smooth Heaviside function since it satisfies $S_{\theta,\tau}(u)=0$ for $u\leq \theta$ and $S_{\theta,\tau}(u)=1$ for $u\geq \theta+\tau.$ 

Note that sigmoidal firing rate functions are very commonly assumed, even in single cell studies, and are more realistic than Heaviside firing rates. The details of $K$ and $S_{\theta,\tau}$ are discussed in \cref{subsec: kernel_constructions,subsec: Firing_Hypoth}. 

Under the normalization $\int_\mathbb{R} K=1$, the parameter assumptions guarantee the existence of three fixed points: $u= 0<u=\beta(\theta,\tau)<u=1$. The fixed points $u= 0$ and $u=1$ are stable, while $u= \beta(\theta,\tau)$ is unstable. We are interested in traveling waves that are heteroclinic orbits connecting the stable fixed points $u= 0$ to $u=1,$ satisfying $U_{\tau}'>0$ when $U_{\tau} \in [\theta,\theta+\tau].$ See Figure \ref{fig: Front_ex2}. This last assumption is critical; it is still a wide open problem as to whether or not traveling fronts exist with disconnected super threshold regions; it is even unsolved in the Heaviside ($\tau=0$) case.
\begin{figure}[H]
\centering
\includegraphics[width=65mm]{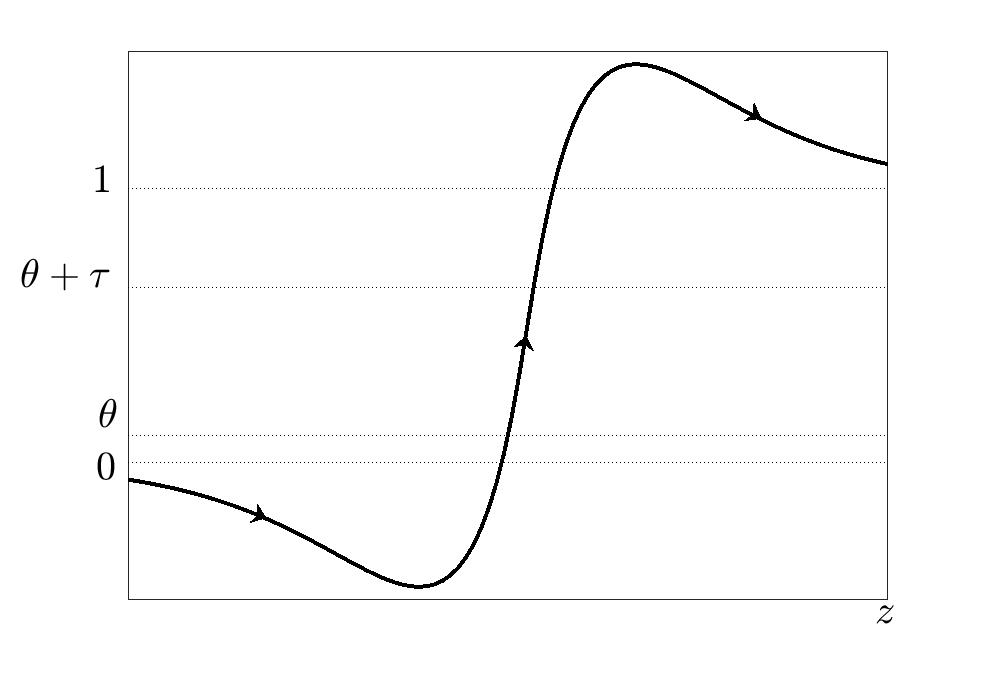}
\caption{Plot of traveling front $U_{\tau}$ with $U_{\tau}'>0$ when $U_{\tau}\in[\theta,\theta+\tau]$.}
\label{fig: Front_ex2}
\end{figure}
While inhibition can prevent wave propagation, we note that the condition $\int_\mathbb{R} K=1>0$ indicates that in totality, inhibition is over-compensated by excitation. This is important in traveling wave studies since one finds from the literature that kernels with inhibition can thwart propagation if excitatory and inhibitory features balance each other out. For example, standing waves \cite{Amari1977,AmariKishimoto,bressloff2005weakly,Multiplebumps-Laing,Botelho2008,pinto2001spatially,ExploitingtheHamiltonian} are commonly of interest. On the other hand, traveling waves exist in \cite{Dyson2019_MBE, LijZhangSolo,LijunZhangetal,magpantay2010wave,Lvwang,Zhang-HowDo,HuttZhang-TravelingWave,ExploitingtheHamiltonian} when $\int_\mathbb{R} K >0$ and Heaviside firing rates are assumed, even though the kernels have a variety of negative regions.

Since real cortical tissue has slower, metabolic features such as spike frequency adaptation, traveling fronts serve as building blocks for traveling pulse solutions to the singularly perturbed system with linear adaptation \cite{PintoandErmentrout-SpatiallyStructuredActivityinSynapticallyCoupledI.TravelingWaves}
\begin{align}
u_t + u+q &= \integ{y}{\mathbb{R}}[]<K(x-y)S_{\theta,\tau}(u(y,t))>, \label{eq: intro_system_1}\\
q_t &=\epsilon (u-\gamma q). \label{eq: intro_system_2}
\end{align}
Here, $0<\epsilon\ll 1.$ The function $q$ is considered to be a slow, leaking current. As was shown in \cite{PintoandErmentrout-SpatiallyStructuredActivityinSynapticallyCoupledI.TravelingWaves}, for $\gamma>0$ sufficiently small, a singular homoclinic orbit can be constructed when $\epsilon=0.$ For $\epsilon>0,$ traveling pulse solutions are true homoclinic orbits that connect the only fixed point $(u,q)\equiv(0,0)$ to itself, crossing each threshold twice. Both crossings occur during the fast transition layers, seen as $O(\epsilon)$ approximations of the traveling front and traveling back (a backwards traveling front with the same wave speed), respectively. 

The front, occurring when $q=\epsilon=0,$ serves the role of the first fast transition layer. Due to related results from \cite{Faye2013,Faye2015}, the existence of a traveling front is nearly sufficient to prove the the existence of a fast pulse. Hence, we dedicate the analysis in this work to the scalar equation \eqref{eq:chSig_main}, with the understanding that pulses also exist using singular perturbation techniques.
\subsection{Approach Based on Previous Results}
\subsubsection*{Heaviside Firing Rates, Lateral Inhibition Kernels}
In the celebrated work of Amari \cite{Amari1977}, it was shown that by assuming firing rates are Heaviside step functions, we can obtain traveling and standing wave solutions with closed form. Such a foundational approach has led to many important discoveries about patterns in neural fields. Moreover, in this case, an Evans function \cite{Zhang-HowDo} can be derived for the stability analysis, leading to analytical and computational methods for assessing the stability of solutions.

For the traveling wave problem, the mathematical approach is fundamentally different for Heaviside ($\tau=0$) versus smooth Heaviside ($\tau>0$) firing rates. In the case of the former, the existence and uniqueness problems do not involve functional analysis methods; model equations reduce to boundary value ODEs. In contrast, the $\tau>0$ case requires careful analysis of linear operators arising from Gateaux derivatives and repeated application of the implicit function theorem in Banach spaces---an approach that is largely inspired by Accinelli \cite{accinelli2010generalization}.

To this end, the most comprehensive traveling wave result concerning Heaviside firing rates and lateral inhibition kernels is from Zhang \cite{Zhang-HowDo}, who proved the existence, uniqueness, and stability under minimal assumptions on $K.$ Recently, existence and uniqueness were proven to hold for any $\theta$ without any assumptions on lateral inhibition kernels \cite{Dyson2019_MBE}. Since the shapes of the $\tau=0$ waves are related to the shapes of those when $\tau>0,$ we will review these results closer in \cref{sec: Apriori_wave}. 

Additionally, Guo \cite{Guo2012} proved the existence and stability of traveling fronts in neural field models with nonsaturating linear gain firing rates and lateral inhibition kernels for the special case where $K$ is the difference of exponentials. Though rigorous, the technique in this work is much different from ours; the specific structure of $K$ and the firing rate allowed the author to reduce the problem to a high order local ODE. In contrast, our result does not depend on reducing the scalar nonlocal equation to a local one. Moreover, our firing rate is smooth and levels off to a saturated value.
\subsubsection*{Nonnegative Kernels, Smooth Firing Rates}
Our approach is most closely inspired by the pioneering study of Ermentrout and McLeod \cite{ExistenceandUniqueness-ErmMcLeod}, where they invoked a continuation argument to prove the existence of unique (modulo translation) monotone front solutions to a similar version of \eqref{eq:chSig_main}. They handled the fact that solutions are translation invariant by using the linearized adjoint space to fix the translation. In their case, the firing rates $S_{\theta,\tau}$ are replaced with more general sigmoidal functions $S\in[0,1]$ with $S'>0$. A large part of their argument utilized the requirement that $K\geq 0$, and therefore, monotone fronts exist. Monotonicity, in particular, greatly simplifies the analysis for uniqueness and stability. Techniques that exploit monotonicity, such as comparison principles \cite{ExistenceUniqueness-Chen, Bates1997}, can be applied. Follow up studies, such as \cite{ermentrout2010stimulus}, have revealed interesting behavior when external stimuli evoke wave propagation.

Here, we also use adjoint spaces as tools to apply a continuation argument. However, our starting point at $\tau=0$ has considerably different dynamics. For one, it is convenient since existence and uniqueness are known to hold. Moreover, the profile of the front and general properties about the wave speed are well understood. Hence, we have valuable information when performing the first application of the implicit function theorem at $\tau=0.$ Regarding the jump itself, some care needs to be taken in the relevant functional spaces due to a formal delta distribution arising from the linearization.

After the first perturbation, we continuously increase $\tau>0,$ As we repeatedly apply the implicit function theorem as $\tau\to \ol{\tau},$ there are difficulties to overcome. Namely, because $K<0$ occurs, the Gateaux derivatives with respect to $u$ are not automatically positive linear operators as they were in \cite{ExistenceandUniqueness-ErmMcLeod}. As a result, proving existence holds under perturbations in $\tau$ via the implicit function theorem is nontrivial. We also have to overcome the fact that the implicit function theorem only guarantees that existence persists following perturbations of $\tau,$ but says nothing about passing through limits as $\tau \to \ol{\tau}.$ We need hypotheses on the wave's profile in order to (at least) pass through subsequential limits via the Arzel\'{a}--Ascoli and Bolzano--Weierstrass theorems.

\subsubsection*{Smooth Heaviside Firing Rates}
The analysis in this paper would be insufficient if our sigmoidal firing rates did not share similarities with Heaviside functions---namely, that $S_{\theta,\tau}(u)=0$ for $u\leq \theta$ and $S_{\theta,\tau}(u)=1$ for $u\geq \theta+\tau.$ Therefore, while the continuation approach in this work is the most similar to \cite{ExistenceandUniqueness-ErmMcLeod}, the core idea of bridging together Zhang's Heaviside result with Ermentrout and McLeod's sigmoidal result was inspired by the work of Coombes and Schmidt \cite{CoombesSchmidt}. They proposed numerical techniques to solve for traveling waves in \eqref{eq:chSig_main}, but left the fixed point arguments open.
\begin{remark} Although the model equations and approach are different, the reader is encouraged to see the work of Bates, Chen, and Chmaj \cite{bates2003traveling,bates2005heteroclinic} for a more broadened literature scope of nonlocal traveling wave and free energy functional problems, respectively. In particular, their work considers nonlocal dynamics where the coupling kernel may have a Mexican hat shape---a feature less often explored in the literature.
\end{remark}

\subsection{Firing Rate Hypotheses}\label{subsec: Firing_Hypoth}
Motivated by \cite{CoombesSchmidt}, we consider $S_{\theta,\tau}$, which is a smoothed Heaviside defined by 
\begin{equation}
\label{eq: smooth_heav}
S_{\theta,\tau}(u) = 
\begin{cases}
0 & u\leq \theta, \\
f(u-\theta, \tau)  & \theta<u<\theta+\tau, \\ 
1 & u\geq  \theta+\tau,
\end{cases}
\end{equation}
where $f(u,\tau)$ is $C^\infty$ smooth and increasing in $u,$ $L^1$ continuous over $u$ with respect to changes in $\tau,$ with $f(0,\tau)=0,$ $f(\tau,\tau)=1.$ See Figure \ref{fig: K and S} (b). We assume $S_{\theta,0}$ is the Heaviside step function in the sense that 
$$
\lim_{\tau \to 0^{+}} \f{\partial f}{\partial u} (u-\theta,\tau)=\delta_\theta(u),
$$
the delta distribution centered at $\theta.$ Note that by definition, for all $u,$
\begin{equation}\label{eq: Firing_bounds}
H(u-(\theta+\tau))\leq S_{\theta,\tau}(u)\leq H(u-\theta)
\end{equation}
holds. The bounds in \eqref{eq: Firing_bounds} play a critical role in proving that for solutions $(U_{\tau},\mu_{\tau}),$ the wave speed $\mu_{\tau}$ is trapped between the unique wave speeds arising from Heaviside models (\cref{lemma: Global_Wave_Bounds}); our other results are guided by this fact.

In addition to monotonicity, we will make a further assumption regarding the symmetry of $S_{\theta,\tau}.$
\begin{itemize}
\item For fixed $\theta$ and $\tau,$ the firing rate $S_{\theta,\tau}(u)$ is odd symmetric about the point $\left(\theta+\f{\tau}{2},S_{\theta,\tau}(\theta+\f{\tau}{2})\right).$  Specifically, for all $u,$
$$
S_{\theta,\tau}\left(\theta+\f{\tau}{2}+u\right)-\f{2}=-\left(S_{\theta,\tau}\left(\theta+\f{\tau}{2}-u\right)-\f{2}\right).
$$
\end{itemize}
Although the main techniques hold without this hypothesis, there are several reasons why we assert it. For one, it guarantees that 
$$
\integ{u}{\theta}[\theta+\tau]<S_{\theta,\tau}(u)>=\f{\tau}{2},
$$
which makes all explicit bounds and formulas in this paper easier to describe. More importantly, odd symmetry guarantees a one-to-one correspondence between the existence of a traveling front and back with the same wave speed. As discussed in \cref{Section: Intro}, this is sufficient to prove the existence of pulse solutions to \eqref{eq: intro_system_1}-\eqref{eq: intro_system_2}.



\subsection{Kernel Hypotheses} \label{subsec: kernel_constructions}

We assume $K\in W^{1,1}(\mathbb{R})$ is continuous on $\mathbb{R}$, smooth on $\mathbb{R}$ except possibly at the origin, and symmetric. Furthermore, in all cases, we assume $K$ has the following typical properties for lateral inhibition kernels:
\begin{itemize}
\item$
 \int_{-\infty}^0 K(x) \, \mathrm{d}x = \int_0^\infty K(x) \, \mathrm{d}x =\f{2}, \qquad |K(\cdot)|\leq C\e{-\rho |\cdot|}$ for some $\rho >0$.
\item There exists $M>0$ such that $K(\cdot)> 0$ on $(-M,M)$ and $K(\cdot)<0$ on $(-\infty,-M)\cup (M,\infty)$.
\end{itemize}
See Figure \ref{fig: K and S} (a). Symmetry is assumed for the mathematical convenience of determining the wave speed sign (positive in this work) in terms of $\theta$ and $\tau$ (see \cref{lemma: positive_wave_speed}).
\begin{figure}[H]
\centering
\subfigure[]
{\includegraphics[width=65mm]{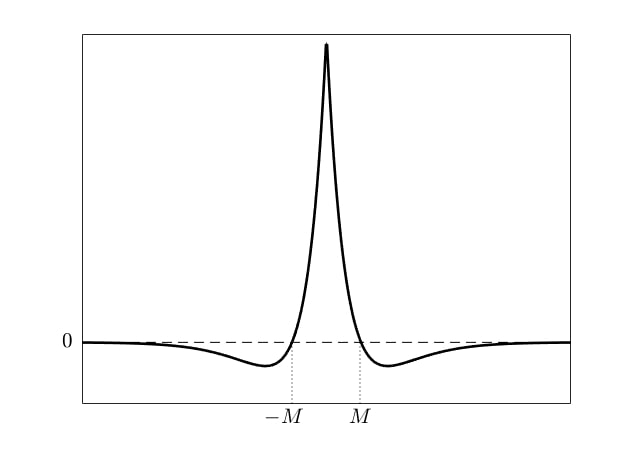}}
\subfigure[]
{\includegraphics[width=65mm]{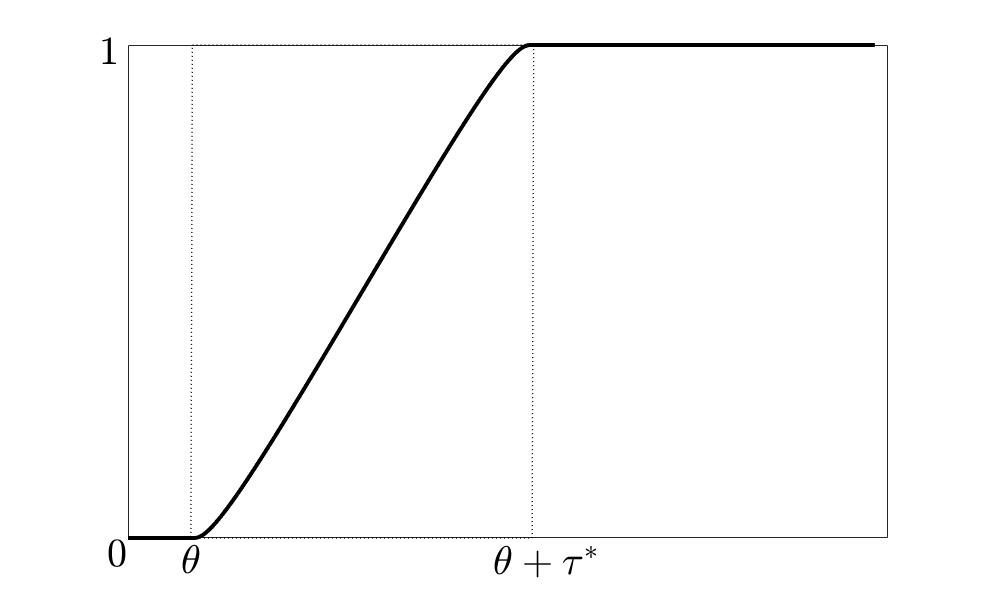}}
\caption{Example of (a) lateral inhibition (Mexican hat) kernel and (b) smooth Heaviside firing rate.}
\label{fig: K and S}
\end{figure}

\section{Main Results}
\label{sec: main_results}
In this section, we state the main theorems that we prove carefully. \cref{sec: Apriori_wave} is dedicated to rigorously proving global wave speed and profile comparisons between solutions arising from smooth versus Heaviside firing rates. The bulk of the functional analysis is performed in \cref{sec: E_fronts}, where we prove that a continuation procedure for existence cannot break down under our hypotheses. Finally, we work through a concrete example in \cref{sec: Example}.
\subsection{Assumptions and Conventions}
\label{subsec: Assumptions}
In addition to the hypotheses discussed in \cref{subsec: kernel_constructions,subsec: Firing_Hypoth}, we make the following simple assumptions.
\begin{itemize}
\item[(A1)] We restrict our attention to leftward traveling fronts by assuming $\theta\in(0,\f{2})$ and $\tau\in[0,1-2\theta]$ always hold. Similar techniques hold for rightward traveling fronts.
\item[(A2)] All solutions $(U_{\tau},\mu_{\tau})$ of interest will satisfy $\mu_{\tau}\geq 0,$ $U_{\tau}'>0$ when $U_{\tau} \in[\theta,\theta+\tau],$ and $U_{\tau}(-\infty)=0, \,U_{\tau}(\infty)=1.$
\item[(A3)] Prior to applying the implicit function theorem about a solution $(U_{\tau},\mu_{\tau}),$ we assume the translation $U_{\tau}^{-1}(\theta)=0$ so that the increasing, inverse function $\Delta_{\tau}\geq 0$ satisfies $U_{\tau}(\Delta_{\tau}(\xi))=\theta+\xi$ for $0\leq\xi\leq\tau$ with $\Delta_{\tau}(0)=0.$ Unless stated otherwise, we assume this translation for technical steps.
\end{itemize}
Under the assumptions above, the solution to $F[U_{\tau},\mu_{\tau},S_{\theta,\tau}](z)=0$ can be written formally as 
$$
U_{\tau}(z)=\int_0^\infty S_{\theta,\tau}(U_{\tau}(y))\left[\f{\mu_{\tau}} \integ{x}{-\infty}[0]<e^{\f{x}{\mu_{\tau}}}K(x+z-y)>\right]\,\mathrm{d}y.
$$
\subsection{Results In Less Technical Terms}\label{subsec: Nontech}
The following are the main highlights of our results. Let $\theta$ and $K$ be fixed.
\begin{itemize}
\item[(1)] {\bf Existence and Uniqueness Perturbation at $\boldsymbol{\tau=0}.$} Let $S_{\theta,0}(u)=H(u-\theta).$ Starting with the unique solution $(V_\theta,v_\theta)=(U_{0},\mu_{0}),$ we show that the implicit function theorem can be applied, perturbing $S_{\theta,0}$ to $S_{\theta,\tau}$ in the $L^1$ norm, and for $\tau\ll 1,$ there exists a unique (modulo translation) solution $(U_{\tau},\mu_{\tau})$ that is a perturbation of $(V_\theta,v_\theta).$ Effectively, this shows that for traveling wave problems where more smoothness on the Heaviside function is desired, we can replace the Heaviside function with a mollifier. We prove the result using a squeeze theorem argument based on the firing rate bounds given in \eqref{eq: Firing_bounds}. 
\item[(2)] {\bf Existence Continuation Criteria ($\boldsymbol{\Delta_{\tau}(\xi)\leq \sigma(\theta)}$) when $\boldsymbol{\tau>0}.$} For an arbitrary solution $(U_{\tau},\mu_{\tau})$, let $\Delta_{\tau}$ be as described in assumption (A3) above. We establish a number $\sigma(\theta)$ such that if $\Delta_{\tau}(\xi)\leq \sigma(\theta),$ we can perturb $\tau$ and the existence of (at least one) solution $(U_{\delta},\mu_{\delta})$ holds for all $\delta\in[\tau,\ol{\tau}).$ Moreover, if $\Delta_{\delta}(\xi)\leq \sigma(\theta)$ for all $\delta\in[\tau,\ol{\tau}),$ we use the Arzel\'{a}--Ascoli and Bolzano--Weierstrass theorems to prove that a solution $(U_{\ol{\tau}},\mu_{\ol{\tau}})$ exists satisfying $\Delta_{\ol{\tau}}(\xi)\leq\sigma(\theta).$ We can then perturb $\ol{\tau}$ and continue the process. Note that since $\Delta_{\tau}$ is increasing on $[0,\tau],$ the condition $\Delta_{\tau}(\xi)\leq \sigma(\theta)$ is equivalent to $\Delta_{\tau}(\tau)\leq \sigma(\theta).$
\item[(3)] {\bf \textit{A priori} Existence Results.} The continuation criteria mentioned above requires ``checking'' whether solutions $(U_{\tau},\mu_{\tau})$ satisfy $\Delta_{\tau}(\xi)\leq \sigma(\theta)$ prior to applying the implicit function theorem or passing families of solutions through subsequences. Since we have not proven uniqueness beyond the $\tau\ll1$ case, as $\tau\to\ol{\tau},$ even if existence holds, we cannot be sure that the continuation of solution pairs forms a unique mapping (modulo translation). Indeed, if uniqueness does not hold, there are possibly many different continuation paths and subsequential limit choices. Hence, it would be satisfying to find some {\it a priori} lower bound $\tau^*(\theta)$ for the set $$\{\tau\in(0,1-2\theta]: \textnormal{At least one solution }(U_{\delta},\mu_{\delta}) \textnormal{ exists for all } \delta\in[0,\tau]\}.$$
We achieve this in \cref{thm: cont_taustar}. Our method is to find global bounds on $\mu_{\tau}$ and $\Delta_{\tau}$ by trapping $S_{\theta,\tau}$ between two Heaviside functions; the details are worked out in technical \cref{lemma: Global_Utau_Bounds,lemma: Global_Wave_Bounds}. We then prove in \cref{thm: cont_taustar} that $\Delta_{\tau}(\tau)>\sigma(\theta)$ cannot occur for $\tau\leq\tau^*(\theta).$
\end{itemize}
\subsection{Results In Technical Terms}
Since our main results are based on comparing solutions to Heaviside solutions, we review Zhang's result \cite{Zhang-HowDo} and establish notation used throughout the paper. Whenever Heaviside solutions are referenced, we always assume the translation $V_\theta(0)=\theta.$
\subsubsection{Review of Heaviside Results}
\label{subsubsec: Review}
In \cite{Zhang-HowDo}, the following result was proved. 
\begin{theorem}
\label{thm: E_and_U_Zhang}
Suppose that $\tau=0,$ $\theta\in(0,\f{2})$ and  $K$ is a lateral inhibition kernel with $\integ{x}{-\infty}[0]<|x|K(x)>\geq 0$. Then there exists a unique (modulo translation) traveling wave front solution $u(x,t) = V_\theta(z) \in C^2(\mathbb{R})$ to \eqref{eq:chSig_main} such that $V_\theta(z) < \theta$ on $(-\infty,0)$, and $V_\theta(z) > \theta$ on $(0,\infty)$. The solution, which has closed form
\begin{align*}
V_\theta(z)&=\f{v_\theta} \int_0^\infty \int_{-\infty}^0 e^{\f{x}{v_\theta}}K(x+z-y)\,\mathrm{d}x\mathrm{d}y \numberthis \label{eq: front_closed} \\
&=\integ{x}{-\infty}[z]<K(x)>-\integ{x}{-\infty}[z]<e^{\f{x-z}{v_\theta}}K(x)>, \\
V_\theta'(z)&= \f{v_\theta} \integ{x}{-\infty}[0]<e^{\f{x}{v_\theta}}K(x+z)> \numberthis \label{eq: front_deriv_closed}
\end{align*}
satisfies the reduced equation
\begin{equation}
\label{eq: front_U0}
v_\theta V_\theta' + V_\theta = \integ{x}{\mathbb{R}}[]<K(z-x)H(V_\theta(x)-\theta)>
\end{equation}
with exponentially decaying limits
\begin{align*}
\lim_{z \to -\infty} V_\theta(z) = 0, \qquad \lim_{z \to \infty} V_\theta(z) = 1, \qquad \lim_{z \to \pm\infty} V_\theta'(z) = 0.
\end{align*}
The wave travels under the traveling coordinate $z = x +v_\theta t$ at the unique wave speed $v_\theta>0$, where $\mu=v_\theta$ is the unique solution to $V_\theta(0;\mu)=\theta$, or equivalently,
\begin{equation} \label{eq: speed_index_heaviside}
\phi(\mu):=\integ{x}{-\infty}[0]<e^{\f{x}{\mu}}K(x)>=\f{2}-\theta.
\end{equation}
\end{theorem}
Note that by \cite{Dyson2019_MBE}, the assumption $\integ{x}{-\infty}[0]<|x|K(x)> \geq 0$ can be removed, and it is proven that in all cases, $\phi'>0$ when $\phi\in(0,\f{2})$. Hence, $v_\theta$ is a well-defined, decreasing function of $\theta.$

For all $\theta,$ the profile of $V_\theta$ on the left half plane is similar in that there exists a unique $\sigma_1(\theta)\in (0,M)$ such that $V_\theta(-\infty)=0,$  $V_\theta$ is decreasing on $(-\infty,-\sigma_1(\theta))$ and increasing on $(-\sigma_1(\theta),0)$ until $V_\theta(0)=\theta.$ 

On the right half plane, there are only two possible wave profiles. One possibility is that $V_\theta$ is increasing on $(0,\infty)$ until $V_\theta(\infty)=1.$ The other is that there exists a unique $z_R(\theta)\in (M,\infty)$ such that $V_\theta$ is increasing on $(0,z_R(\theta))$ and decreasing on $(z_R(\theta),\infty)$ until $V_\theta(\infty)=1.$ 

Importantly, we note that $V_\theta'(z)>0$ on $(-\sigma_1(\theta),\sigma_1(\theta)).$
\subsubsection{Main Technical Statements}
The first two technical lemmas (using the notation from \cref{thm: E_and_U_Zhang}) establish important global bounds on $\mu_{\tau}$ and $\Delta_{\tau}.$ These bounds are critical for proving uniqueness and for formulating our continuation results in \cref{thm: continuation,thm: cont_taustar}. We prove them in \cref{sec: Apriori_wave}.
\newtheorem*{theorem*}{Theorem}
\begin{lemma}[Global Wave Speed Bounds]
\label{lemma: Global_Wave_Bounds}
Suppose $\tau\in(0,1-2\theta]$ and $(U_{\tau},\mu_{\tau})$ is a traveling front solution to \eqref{eq:chSig_main} with firing rate $S_{\theta,\tau}(u).$ Then
\begin{itemize}
\item[(i)] $0\leq \mu_{\tau}<v_\theta.$
\item[(ii)] If additionally $\tau<\f{2}-\theta,$ then $v_{\theta+\tau}<\mu_{\tau}<v_\theta.$ 
\end{itemize}
\end{lemma}

\begin{lemma}[Global Inverse Bounds]
\label{lemma: Global_Utau_Bounds}
Suppose $(U_{\tau},\mu_{\tau})$ is a traveling front solution to \eqref{eq:chSig_main} with firing rate $S_{\theta,\tau}(u).$ Then for all $\xi\in(0,\tau],$ the estimate $\omega_L(\Delta_{\tau}(\xi),\tau)<\xi<\omega_U(\Delta_{\tau}(\xi),\tau)$ holds, where 
$$
\omega_L(a,\tau)=
\begin{cases}
\displaystyle
\min_{\substack{y \in [0,a] \\ \delta\in [\theta,\theta+\tau]}} (V_\delta(a-y)-V_\delta(-y)), &\tau<\f{2}-\theta \\
\displaystyle\min_{\substack{y \in [0,a] \\ \delta\in [\theta,\f{2}]}} (V_\delta(a-y)-V_\delta(-y)), &\tau\geq\f{2}-\theta
\end{cases}
$$
$$
\omega_U(a,\tau)=
\begin{cases}
\displaystyle
\max_{\substack{y \in [0,a] \\ \delta\in [\theta,\theta+\tau]}} (V_\delta(a-y)-V_\delta(-y)), &\tau<\f{2}-\theta \\
\displaystyle\max_{\substack{y \in [0,a] \\ \delta\in [\theta,\f{2}]}} (V_\delta(a-y)-V_\delta(-y)), &\tau\geq\f{2}-\theta
\end{cases}
$$
%
%
\end{lemma}
We see in \cref{lemma: Global_Utau_Bounds} that the bounds on $\Delta_{\tau}$ are defined implicitly. However, one of the conveniences of the result is that all functions $V_\delta$ have closed form.

Using the global bounds above, we are able to obtain our main existence and uniqueness results. The next theorem rigorously shows that fronts persist under $L^1$ perturbations of the Heaviside firing rate. We use the implicit function theorem to prove existence and \cref{lemma: Global_Wave_Bounds,lemma: Global_Utau_Bounds} to prove uniqueness. The application of this theorem gives the results in \cref{subsec: Nontech} (1). 
\begin{theorem}[Perturbation of Heaviside]
\label{thm: perturb}
Suppose $(V_\theta,v_\theta)$ is a unique front solution pair to 
\begin{equation}
\label{eq: front_U0}
v_\theta V_\theta' + V_\theta = \integ{y}{\mathbb{R}}[]<K(z-y)S_{\theta,0}(V_\theta(y))>,
\end{equation}
with $S_{\theta,0}(u)=H(u-\theta)$ and $(V_\theta,v_\theta)$ satisfying all of the conditions of \cref{thm: E_and_U_Zhang}. Assume $\{S_{\theta,\tau} \}_{\tau\geq 0}$ is a family of smooth firing rate functions continuously deformed in $\norm{\cdot}_1$ by $\tau$. Then there exists small $\delta_0>0$ such that for $\tau \in [0,\delta_0)$, there exists a unique (modulo translation) solution pair $(U_{\tau},\mu_{\tau})$ satisfying
\begin{align}
\label{eq: front_Utau}
\mu_{\tau} U_{\tau}' + U_{\tau} = \integ{y}{\mathbb{R}}[]<K(z-y)S_{\theta,\tau}(U_{\tau}(y))>.
\end{align}
The solution $U_{\tau} \in C^\infty(\mathbb{R})$ is translation invariant, travels under the coordinate $z=x+\mu_{\tau} t$ with wave speed $\mu_{\tau}>0$, and satisfies (A2) with $\Delta_{\tau}(\xi)\ll1.$ Moreover, the limits 
\begin{equation}\label{eq: Utau_limits}
\lim_{z\to -\infty} U_{\tau}(z)=0, \qquad
\lim_{z\to \infty} U_{\tau}(z)=1, \qquad
\lim_{z\to \pm\infty} U_{\tau}^{(j)}(z)=0, \qquad \textnormal{ for } j\geq 1
\end{equation}
hold. Under the translation $U_{\tau}^{-1}(\theta)=0,$ the solution takes on the form
$$
U_{\tau}(z)=\int_0^\infty S_{\theta,\tau}(U_{\tau}(y))\left[\f{\mu_{\tau}} \integ{x}{-\infty}[0]<e^{\f{x}{\mu_{\tau}}}K(x+z-y)>\right]\,\mathrm{d}y.
$$
\end{theorem}
\subsection*{Continuation Criteria and Calculation of $\sigma(\theta)$}
\label{subsec: Continuation_Hypothesis}
Following up our perturbation of $\tau=0$ result, we assume $\tau>0$ for the remaining results, which help us achieve \cref{subsec: Nontech} (2)-(3). We calculate $\sigma(\theta)$ and prove that when $\Delta_{\tau}(\xi)\leq\sigma(\theta)$ the implicit function, Arzel\'{a}--Ascoli, and Bolzano--Weierstrass theorems can be applied, generating a continuation procedure that cannot fail to pass through limits (of subsequences).

Fix $\theta\in(0,\f{2}).$ As is discussed in \cref{subsubsec: Review}, define $\sigma_1(\theta)\in (0,M)$ to be the unique constant such that $V_\theta'(-\sigma_1(\theta))=0,$ or equivalently,
$$
\integ{x}{-\infty}[0]<e^{\f{x}{v_\theta}}K(x-\sigma_1(\theta))>=0.
$$

Define $\sigma_2(\theta)$ to be the positive constant that is the unique solution to the equation
\begin{align}
\left(\int_{-\infty}^{-M-\sigma_2(\theta)}+\int_{-M}^{-M+\sigma_2(\theta)}\right) K(x) \, \mathrm{d}x &=\theta. \label{eq: sigma2}
\end{align}
Note that $\sigma_2(\theta)$ is an increasing function on $(0,\f{2}).$ With these constants defined, let 
$$
\sigma(\theta):=\min\{\sigma_1(\theta),\,\sigma_2(\theta)\}.
$$ 
Some remarks about the strategy behind demanding $\Delta_{\tau}(\xi)\leq \sigma_1(\theta)$ and $\Delta_{\tau}(\xi)\leq \sigma_2(\theta)$ are in order. Both bounds are set based on what can go wrong for limiting solutions as $\tau\to\ol{\tau}.$ Namely, from assumption (A2), we want $U_{\ol{\tau}}$ to be increasing through the threshold region $[\theta,\theta+\ol{\tau}]$ with $U_{\ol{\tau}}(-\infty)=0, \, U_{\ol{\tau}}(\infty)=1.$

Firstly, in the limit, we need to show that if $U_{\ol{\tau}}'\geq (\not\equiv) 0$ when $U_{\ol{\tau}}\in(\theta,\theta+\ol{\tau}),$ then $U_{\ol{\tau}}'>0$ in the same region. This is analogous to \cite{ExistenceandUniqueness-ErmMcLeod} and is actually trivial to show when $K>0.$ However, when $K<0$ occurs, this is less obvious. Requiring $\Delta_{\tau}(\xi)\leq \sigma_1(\theta)$ ensures this fact, as well as the fact that the implicit function theorem can be applied because $\lambda=0$ is a simple eigenvalue of the Gateaux derivative.

Secondly, we need to be guaranteed that on the left (resp. right) half plane, local maximums (resp. minimums) do not touch the threshold region boundary $\{\theta,\theta+\ol{\tau}\}.$ The parameter $\sigma_2(\theta)$ is calculated for this purpose. By setting $\sigma(\theta)=\min\{\sigma_1(\theta),\sigma_2(\theta)\},$ all threshold conditions are met when $\Delta_{\tau}(\xi)\leq \sigma(\theta).$

Finally, as will be shown in \cref{lemma: U_uniformly_bounded}, uniformly bounding $\Delta_{\tau}(\xi)$ is sufficient to show that if solutions exist on $[\tau,\ol{\tau}),$ then there exists a subsequence of solutions with $(U_{\delta_n},\mu_{\delta_n})\to (U_{\ol{\tau}},\mu_{\ol{\tau}})$ as $\delta_n\to \ol{\tau}$ with $U_{\ol{\tau}}(-\infty)=0, \, U_{\ol{\tau}}(\infty)=1$ holding.

The following theorem summarizes these remarks and allows us to achieve \cref{subsec: Nontech} (2).
\begin{theorem}[Continuation Criteria]
\label{thm: continuation}
Suppose at $\tau\in(0,1-2\theta),$ a solution $(U_{\tau},\mu_{\tau})$ exists satisfying assumption (A2) with $\Delta_{\tau}(\xi)\leq \sigma(\theta).$ Then there exists $\ol{\tau}>\tau$ such that for all $\delta\in[\tau,\ol{\tau}),$ a solution $(U_{\delta},\mu_{\delta})$ exists. Moreover, if for all $\delta\in[\tau,\ol{\tau}),$ solutions satisfy (A2) with $\Delta_{\delta}(\xi)\leq \sigma(\theta),$  then there exists a solution $(U_{\ol{\tau}},\mu_{\ol{\tau}})$ satisfying (A2) with $\Delta_{\ol{\tau}}(\xi)\leq \sigma(\theta).$
\end{theorem}

\subsubsection{A Priori Existence Results}
We are left with the {\it a priori} existence result discussed in \cref{subsec: Nontech} (3). In \cref{thm: cont_taustar} below, we find a lower bound $\tau^*(\theta)$ for the set $$\{\tau\in(0,1-2\theta]: \textnormal{At least one solution }(U_{\delta},\mu_{\delta}) \textnormal{ exists for all } \delta\in[0,\tau]\}.$$
We prove it using the global bounds presented in \cref{lemma: Global_Utau_Bounds,lemma: Global_Wave_Bounds}.

Define the function 
\begin{equation}\label{eq: Phi_tau}
\Phi(\theta,\tau):=\tau-\omega_L(\sigma(\theta),\tau).
\end{equation}
By the definition in \cref{lemma: Global_Utau_Bounds}, we see that $\omega_L(\sigma(\theta),\tau)$ is a decreasing function in $\tau$ for $\tau \in(0,\f{2}-\theta)$ and constant for $\tau\in[\f{2}-\theta,1-2\theta].$ Therefore, $\Phi(\theta,\tau)$ is increasing in $\tau$ with $\Phi(\theta,0)=-\omega_L(\sigma(\theta),0)<0.$ We use this fact to prove the following theorem. Our proof assumes \cref{lemma: Global_Utau_Bounds,lemma: Global_Wave_Bounds} and \cref{thm: perturb,thm: continuation} hold, even though we still have to prove them separately with techniques independent of this theorem.
\begin{theorem}[Existence for $\tau\leq \tau^*(\theta)$]\label{thm: cont_taustar}
If $\Phi(\theta,1-2\theta)>0,$ define $\tau^*(\theta) \in (0,1-2\theta]$ by $\tau^*(\theta)=\inf\{\tau\in(0,1-2\theta]: \Phi(\theta,\tau)>0\}.$ If $\Phi(\theta,1-2\theta)\leq 0,$ set $\tau^*(\theta)=1-2\theta.$ Then for all $\tau\leq \tau^*(\theta),$ there exists a solution $(U_{\tau},\mu_{\tau})$ satisfying (A2) with $\Delta_{\tau}(\xi)\leq\sigma(\theta).$
\end{theorem}
\begin{proof}
By \cref{lemma: Global_Utau_Bounds}, we know $\omega_L(\Delta_{\tau}(\tau),\tau)<\tau$ has to hold for any solution $(U_{\tau},\mu_{\tau})$ satisfying (A2). We also know by \cref{thm: perturb,thm: continuation} that existence may only fail at $\tau$ if we cannot find solutions where $\Delta_{\tau}(\tau)\leq\sigma(\theta).$

Suppose $\Phi(\theta,1-2\theta)>0.$ Then 
$$\{\tau\in(0,1-2\theta]: \Phi(\theta,\tau)>0\}\neq \emptyset$$ and $\tau^*(\theta)\geq 0$ exists. We show $\tau^*(\theta)>0.$ Suppose $\tau^*(\theta)=0.$ By definition, any sufficiently small $\epsilon$ is not a lower bound for the set above so there exists $\tau_1\in[\tau^*(\theta),\epsilon)$ such that $\Phi(\theta,\tau_1)>0,$ or equivalently, $0<\omega_L(\sigma(\theta),\tau_1)<\tau_1.$ By the squeeze theorem, this implies 
$$
\lim_{\tau_1\to0} \omega_L(\sigma(\theta),\tau_1)=\lim_{\tau_1\to0}
\min_{\substack{y \in [0,\sigma(\theta)] \\ \delta\in [\theta,\theta+\tau_1]}} (V_\delta(\sigma(\theta)-y)-V_\delta(-y))=0.
$$ 
Based on the fact that each $V_\delta(y)$ is increasing on $[-\sigma(\theta),\sigma(\theta)]$ (see \cref{cor: r_positive}), this is impossible for fixed $\sigma(\theta)>0.$ Hence, $\tau^*(\theta)>0$.

Suppose for some continuation process, as $\tau\to\ol{\tau},$ the value $\Delta_{\tau_0}(\tau_0)=\sigma(\theta)$ occurs for some $\tau_0.$ From the global bound $\omega_L(\Delta_{\tau_0}(\tau_0),\tau_0)<\tau_0,$ it follows that $\Phi(\theta,\tau_0)>0$ so $\tau_0\geq\tau^*(\theta)>0.$ Hence, for all $\tau<\tau^*(\theta),$ we know that all solutions satisfy $\Delta_{\tau}(\tau)<\sigma(\theta).$ By \cref{thm: continuation}, we may take any continuation path as $\tau\to\tau^*(\theta)$ and at least one solution will exist for all $\tau\leq\tau^*(\theta)$ with $\Delta_{\tau}(\xi)\leq \sigma(\theta).$

Now suppose $\Phi(\theta,1-2\theta)\leq 0.$ Since $\Phi$ is increasing in $\tau,$ this means $\Phi(\theta,\tau)\leq 0$ for all $\tau\leq 1-2\theta.$ By similar reasoning to the first case, at least one solution will exist for all $\tau\in(0,1-2\theta]$ since $\Delta_{\tau}(\tau)>\sigma(\theta)$ cannot occur.
\end{proof}
\begin{corollary}
If $\Phi(\theta,1-2\theta)\leq 0,$ then there exists a standing front $(U_{\tau},0)$ satisfying (A2) when $\tau=1-2\theta.$
\end{corollary}
\begin{proof}
By \cref{thm: cont_taustar}, if $\Phi(\theta,1-2\theta)\leq 0,$ continuation may proceed from $\tau=0$ to $\tau=1-2\theta,$ producing fronts satisfying (A2). By \cref{lemma: positive_wave_speed} below, $\mu_{\tau}=0$ when $\tau=1-2\theta.$
\end{proof}
\begin{remark}
The proof of \cref{thm: cont_taustar} assumed $\Phi$ is increasing in $\tau,$ but not necessarily continuous. Since $\omega_L(\sigma(\theta),\tau)$ is constant for $\tau\geq \f{2}-\theta,$ we know that $\Phi$ is continuous on this domain. However, we have not proven that $\omega_L(\sigma(\theta),\tau)$ is continuous for $\tau<\f{2}-\theta.$ For this reason, we needed to account for jumps over zero as $\Phi$ increases.
\end{remark}
\begin{remark}
To simplify the technical steps, we always assume that the lower threshold $\theta$ is fixed, and the upper threshold $\theta+\tau$ is increasing through the parameter $\tau$. However, one can easily set up the scheme with the upper threshold $\ol{\theta}$ fixed, with the lower threshold $\ol{\theta}-\tau$ decreasing. The motivation for choosing the former over the latter is somewhat arbitrary. We suspect that the consideration of both schemes may play a role in proving uniqueness beyond the perturbative case.
\end{remark}

\section{Proof of Global Bounds}\label{sec: Apriori_wave}
In this section, we prove \cref{lemma: Global_Wave_Bounds,lemma: Global_Utau_Bounds}. We achieve this by drawing important connections between Heaviside solutions and smooth Heaviside solutions.

When $\tau>0,$ we no longer have a simple formula for the wave speed. We can, however, show that in order for $\mu_{\tau}\geq0$, we need $\tau\leq1-2\theta.$
\begin{lemma}\label{lemma: positive_wave_speed}
Under assumption (A2), the wave speed $\mu_{\tau}$ is nonnegative if $\tau\in(0,1-2\theta].$
\end{lemma}
\begin{proof}
In \eqref{eq: front_Utau}, multiply both sides by $S_{\theta,\tau}'(U_{\tau}(z))U'_{\tau}(z)$ and integrate over all $z$. Since $K$ is symmetric and $\int_\mathbb{R} K=1$, the same simplification as in \cite[Theorem 3.1]{ExistenceandUniqueness-ErmMcLeod} applies in that the wave speed has the formula
\begin{align*}
\mu_{\tau}&=\f{\integ{y}{\mathbb{R}}[]<(-U_{\tau}(y)+S_{\theta,\tau}(U_{\tau}(y)))S_{\theta,\tau}'(U_{\tau}(y))U_{\tau}'(y) >}{\integ{x}{\mathbb{R}}[]<U_{\tau}'(x)^2 S_{\theta,\tau}'(U_{\tau}(x))>} \\
&= \f{\integ{y}{0}[\Delta_{\tau}(\tau)]<(-U_{\tau}(y)+S_{\theta,\tau}(U_{\tau}(y)))S_{\theta,\tau}'(U_{\tau}(y))U_{\tau}'(y) >}{\integ{x}{0}[\Delta_{\tau}(\tau)]<U_{\tau}'(x)^2 S_{\theta,\tau}'(U_{\tau}(x))>},\numberthis \label{eq: wave_formula}
\end{align*}
where the changes in integration bounds are justified since $S_{\theta,\tau}'(U_{\tau})=0$ when $U_{\tau}\leq \theta$ or $U_{\tau}\geq \theta+\tau$.
Clearly the denominator is positive so the proof comes down to simplifying the numerator.

By the definition of $S_{\theta,\tau}$ and since $U_{\tau}'(\cdot)>0$ on $[0,\Delta_{\tau}(\tau)]$, a simple calculation shows that the numerator of \eqref{eq: wave_formula} can be written as
\begin{align*}
&\phantom{=}\integ{u}{\theta}[\theta+\tau]<(-u+S_{\theta,\tau}(u))S_{\theta,\tau}'(u)> \\
&= 1-(\theta+\tau)-\integ{u}{\theta}[\theta+\tau]<(-1+S_{\theta,\tau}'(u))S_{\theta,\tau}(u)> \\
&=1-(\theta+\tau)+\integ{u}{\theta}[\theta+\tau]<S_{\theta,\tau}(u)>-\f{2}\integ{u}{\theta}[\theta+\tau]<\left[S_{\theta,\tau}(u)^2\right] '> \\
&= \f{2} -(\theta+\tau)+\integ{u}{\theta}[\theta+\tau]<S_{\theta,\tau}(u)> \\
&=  \f{2} -\left(\theta+\f{\tau}{2}\right)
\end{align*}
by the assumption $\integ{u}{\theta}[\theta+\tau]<S_{\theta,\tau}(u)>=\f{\tau}{2}.$ The result follows immediately.
\end{proof}
\subsection{Wave Speed Bounds}
\label{subsec: Wave_Speed_Bounds}
\subsubsection{Proof of \cref{lemma: Global_Wave_Bounds} (i)}
We show that $\mu_{\tau}<v_\theta,$ where $v_\theta$ solves
$$
\integ{x}{-\infty}[0]<e^{\f{x}{v_\theta}}K(x)>=\f{2}-\theta.
$$
Define the auxiliary function $\Gamma(z,\mu)$ and its derivative by
\begin{align}
\Gamma(z,\mu)&=\f{\mu}\int_0^\infty \int_{-\infty}^0 e^{\f{x}{\mu}}K(x+z-y)\,\mathrm{d}x\mathrm{d}y, \\
\Gamma_z(z,\mu)&= \f{\mu}\integ{x}{-\infty}[0]<e^{\f{x}{\mu}}K(x+z)>,
\end{align}
for $z\in\R, \mu>0.$ Note that $\Gamma(-\infty,\mu)=0,\,\Gamma(\infty,\mu)=1,$ and in fact, $\Gamma$ is equivalent to the formal solution $V_{\ol{\theta}}(z)$ when $\mu=v_{\ol{\theta}}$ for some $\ol{\theta}\in (\theta,\f{2}),$ which is what we are trying to prove.

Suppose $(U_{\tau},\mu_{\tau})$ is a solution corresponding with firing rate $S_{\theta,\tau}.$ Then since $(U_{\tau},\mu_{\tau})$ solves
$$
\mu_{\tau} U_{\tau}' + U_{\tau}=\integ{y}{0}[\infty]<K(z-y)S_{\theta,\tau}(U_{\tau}(y))>
$$
and $S_{\theta,\tau}(U_{\tau}(y))=1$ for $y\geq \Delta_{\tau}(\tau),$ it follows that
\begin{align*}
U_{\tau}(z) &=\f{\mu_{\tau}}\int_0^\infty S_{\theta,\tau}(U_{\tau}(y))\left[\int_{-\infty}^0 e^{\f{x}{\mu_{\tau}}}K(x+z-y)\,\mathrm{d}x\right]\mathrm{d}y \\
&= \integ{y}{0}[\infty]<S_{\theta,\tau}(U_{\tau}(y))\Gamma_y(z-y,\mu_{\tau})> \\
&= \integ{y}{0}[\Delta_{\tau}(\tau)]<S_{\theta,\tau}(U_{\tau}(y))\Gamma_y(z-y,\mu_{\tau})>+\integ{y}{\Delta_{\tau}(\tau)}[\infty]<\Gamma_y(z-y,\mu_{\tau})> \\
&= \integ{y}{0}[\Delta_{\tau}(\tau)]<S_{\theta,\tau}(U_{\tau}(y))\Gamma_y(z-y,\mu_{\tau})> + \Gamma(z-\Delta_{\tau}(\tau),\mu_{\tau}).
\end{align*}
Plugging in $z=0,$ we have 
\begin{equation}\label{eq: theta_global}
\theta= \integ{y}{0}[\Delta_{\tau}(\tau)]<S_{\theta,\tau}(U_{\tau}(y))\Gamma_y(-y,\mu_{\tau})> + \Gamma(-\Delta_{\tau}(\tau),\mu_{\tau}).
\end{equation}
We consider the behavior of $\Gamma_y(-y,\mu_{\tau})$ for $y\in(0,\infty).$ By considering the function $h(y)=e^{\f{-y}{\mu_{\tau}}}\Gamma_y(-y,\mu_{\tau}),$ which has the same sign as $\Gamma_y(-y,\mu_{\tau}),$ we see
$$
h(y)=\f{\mu_{\tau}}\integ{x}{-\infty}[-y]<e^{\f{x}{\mu_{\tau}}}K(x)>.
$$
Recalling that $K(x)<0$ on $(-\infty,-M)$ and $K(x)>0$ on $(-M,0),$ it follows that $h(y)$ is decreasing on $(0,M)$ and increasing on $(M,\infty).$ Hence, prior to knowing information about $\Delta_{\tau}(\tau),$ we have three cases to consider.

The first case is that $\Gamma_y(-y,\mu_{\tau})<0$ for $y\in(0,\infty).$ We can rule this case out since if this were the case, then since $\Gamma(-\infty,\mu_{\tau})=0,$ this implies $\Gamma(-\Delta_{\tau}(\tau),\mu_{\tau})<0.$ Then both terms on the right hand side of \eqref{eq: theta_global} are clearly negative, which is a contradiction.

The second case is that there exists a critical point $z_c\in[\Delta_{\tau}(\tau),\infty)$
such that $\Gamma_y(-y,\mu_{\tau})<0$ for $y\in(z_c,\infty)$ and $\Gamma_y(-y,\mu_{\tau})>0$ for $y\in(0,z_c).$
Clearly, $\Gamma_y(-y,\mu_{\tau})>0$ for $y\in(0,\Delta_{\tau}(\tau)).$ Therefore, in \eqref{eq: theta_global}, we may use the bound $S_{\theta,\tau}\leq 1$ on the positive integrand. After integrating, we conclude that
\begin{align*}
\theta < \integ{y}{0}[\Delta_{\tau}(\tau)]<\Gamma_y(-y,\mu_{\tau})> + \Gamma(-\Delta_{\tau}(\tau),\mu_{\tau})=\Gamma(0,\mu_{\tau}).
\end{align*}
If this is a contradiction, then this case can be ruled out. If not, then since $V_\theta(0)=\theta,$ it follows that 
$$
\f{v_\theta} \int_0^\infty \int_{-\infty}^0 e^{\f{x}{v_\theta}}K(x-y)\,\mathrm{d}x\mathrm{d}y<\f{\mu_{\tau}} \int_0^\infty \int_{-\infty}^0 e^{\f{x}{\mu_{\tau}}}K(x-y)\,\mathrm{d}x\mathrm{d}y,
$$ 
or equivalently, after integrating by parts in $y,$ we have $\phi(\mu_{\tau})<\phi(v_\theta)=\f{2}-\theta,$ where $\phi$ is the speed index function from \eqref{eq: speed_index_heaviside}. Since $\phi$ is strictly increasing when $\phi\in(0,\f{2}),$ we conclude that $\mu_{\tau}<v_\theta.$

The third case is the same as case two, except this time, the critical point satisfies $z_c\in(0,\Delta_{\tau}(\tau)).$ Then in the integral in \eqref{eq: theta_global}, we use the fact that $S_{\theta,\tau}(U_{\tau}(y))$ is an increasing function for $y\in(0,\Delta_{\tau}(\tau)).$ On $(-z_c,0),$ we see $\Gamma_y(-y,\mu_{\tau})>0$ and on $(-\Delta_{\tau}(\tau),-z_c),$ we see $\Gamma_y(-y,\mu_{\tau})<0$  so on both intervals,
\begin{equation}\label{eq: S_increasing_bound}
S_{\theta,\tau}(U_{\tau}(y))\Gamma_y(-y,\mu_{\tau})<S_{\theta,\tau}(U_{\tau}(z_c))\Gamma_y(-y,\mu_{\tau}).
\end{equation}
It follows that
\begin{equation}\label{eq: theta_global_inequality}
\begin{split}
\theta&< S_{\theta,\tau}(U_{\tau}(z_c))\integ{y}{0}[\Delta_{\tau}(\tau)]<\Gamma_y(-y,\mu_{\tau})> + \Gamma(-\Delta_{\tau}(\tau),\mu_{\tau}) \\
&=S_{\theta,\tau}(U_{\tau}(z_c))\left[\Gamma(0,\mu_{\tau})-\Gamma(-\Delta_{\tau}(\tau),\mu_{\tau})\right] + \Gamma(-\Delta_{\tau}(\tau),\mu_{\tau}). 
\end{split}
\end{equation}
Obviously, $\Gamma(-\Delta_{\tau}(\tau),\mu_{\tau})<0$ so if $\Gamma(0,\mu_{\tau})<\Gamma(-\Delta_{\tau}(\tau),\mu_{\tau}),$ we arrive at a contradiction in \eqref{eq: theta_global_inequality}. Assuming $\Gamma(0,\mu_{\tau})\geq\Gamma(-\Delta_{\tau}(\tau),\mu_{\tau}),$ we may use the bound $S_{\theta,\tau}\leq 1$ to conclude that $\theta<\Gamma(0,\mu_{\tau}).$ Either this is a contradiction and this case is ruled out, or we apply similar analysis that was used in case two to conclude that $\mu_{\tau}<v_\theta.$

We have completed the proof of the first part of \cref{lemma: Global_Wave_Bounds}, showing that $\mu_{\tau}<v_\theta.$ 
\begin{corollary}\label{cor: speed_equiv}
Let $(U_{\tau},\mu_{\tau})$ be a traveling wave solution satisfying assumption (A2). Then there exists a unique $\ol{\theta}\in(\theta,\f{2})$ such that $\mu_{\tau}=v_{\ol{\theta}}$ and $\Gamma(z,\mu_{\tau})=V_{\ol{\theta}}(z).$
\end{corollary}
\begin{proof}
By \cref{lemma: Global_Wave_Bounds} (i), we know $\mu_{\tau}<v_\theta.$ Since $\phi$ is increasing when $\phi\in(0,\f{2})$ and $v_\theta$ satisfies $\phi(v_\theta)=\f{2}-\theta,$ there exists a unique $\ol{\theta}\in (\theta,\f{2})$ such that $\phi(\mu_{\tau})=\f{2}-\ol{\theta},$ which implies $\mu_{\tau}=v_{\ol{\theta}}.$ The fact that $\Gamma(z,\mu_{\tau})=V_{\ol{\theta}}(z)$ follows directly from the definitions.
\end{proof}

\subsubsection{Proof of \cref{lemma: Global_Wave_Bounds} (ii)}
We are trying to show that if $\tau<\f{2}-\theta,$ then $\mu_{\tau}>v_{\theta+\tau},$ where $v_{\theta+\tau}$ is the unique wave speed corresponding with firing rate $H(u-(\theta+\tau)).$ In principle, the proof technique is extremely similar to the last one; we outline the key differences.

First of all, now that we know $\mu_{\tau}<v_\theta,$ we are able to make stronger assertions regarding the shape of $\Gamma(z,\mu_{\tau}).$
By \cref{cor: speed_equiv}, there exists a unique $\ol{\theta}\in (\theta,\f{2})$ such that $\mu_{\tau}=v_{\ol{\theta}}.$ Then $\Gamma(z,\mu_{\tau})=V_{\ol{\theta}}(z),$ where the profile is understood from \cref{subsubsec: Review}.

Given the solution $(U_{\tau},\mu_{\tau}),$ this time translate so that $U_{\tau}^{-1}(\theta+\tau)=0.$ Since $U_{\tau}'>0$ when $U_{\tau}\in[\theta,\theta+\tau],$ we may denote $-\Delta_{\tau}(\tau)=U_{\tau}^{-1}(\theta)<0.$ Then by plugging in $z=0,$ we have 
\begin{equation}\label{eq: thetatau_global}
\begin{split}
\theta+\tau&=\integ{y}{-\Delta_{\tau}(\tau)}[0]<S_{\theta,\tau}(U_{\tau}(y))V_{\ol{\theta}}'(-y)> + V_{\ol{\theta}}(0) \\
&=\integ{y}{0}[\Delta_{\tau}(\tau)]<S_{\theta,\tau}(U_{\tau}(-y))V_{\ol{\theta}}'(y)> + \ol{\theta}.
\end{split}
\end{equation}
As outlined in \cref{thm: E_and_U_Zhang}, there are two possibilities based on the shape of $V_{\ol{\theta}}$ on the right half plane. The first is that $V_{\ol{\theta}}'(y)>0$ for $y\in(0,\Delta_{\tau}(\tau)).$ If this is the case, then the integral in \eqref{eq: thetatau_global} is positive so we may conclude that $\theta+\tau>\ol{\theta}.$ Since $\phi$ is increasing and $\phi(v_{\theta+\tau})=\f{2}-(\theta+\tau)$, $\phi(v_{\ol{\theta}})=\f{2}-\ol{\theta},$ it follows that $\mu_{\tau}>v_{\theta+\tau}.$

The second possibility is there exists a critical point $z_c\in(0,\Delta_{\tau}(\tau))$ such that $V_{\ol{\theta}}'(y)>0$ for $y\in(0,z_c)$ and $V_{\ol{\theta}}'(y)<0$ for $y\in(z_c,\Delta_{\tau}(\tau)).$ Since $S_{\theta,\tau}(U_{\tau}(-y))$ is a decreasing function for $y\in(0,\Delta_{\tau}(\tau)),$ it follows from similar reasoning to \eqref{eq: S_increasing_bound} that
\begin{equation}
\begin{split}
\theta+\tau&>S_{\theta,\tau}(U_{\tau}(-z_c))\integ{y}{0}[\Delta_{\tau}(\tau)]<V_{\ol{\theta}}'(y)> + \ol{\theta}\\
&= S_{\theta,\tau}(U_{\tau}(-z_c))\left[V_{\ol{\theta}}(\Delta_{\tau}(\tau))-V_{\ol{\theta}}(0)\right] + \ol{\theta}.
\end{split}
\end{equation}
Since $V_{\ol{\theta}}(z)>V_{\ol{\theta}}(0)=\ol{\theta}$ for $z\in(0,\infty),$ we may conclude again that $\theta+\tau>\ol{\theta}.$ The same analysis as in the first case shows that $\mu_{\tau}>v_{\theta+\tau}.$ This completes the proof of \cref{lemma: Global_Wave_Bounds} (ii).

\subsection{Inverse Bounds}
Our main goal in this subsection is to prove \cref{lemma: Global_Utau_Bounds}, establishing global bounds on the inverse function $\Delta_{\tau}(\xi)$ when $U_{\tau}\in[\theta,\theta+\tau].$

Fix a solution $(U_{\tau},\mu_{\tau})$ with the translation $U_{\tau}^{-1}(\theta)=0.$ If $\tau\in(0,\f{2}-\theta),$ then based on \cref{lemma: Global_Wave_Bounds}, we may assume $$
v_{\theta+\tau}<\mu_{\tau}<v_\theta.
$$
By \cref{cor: speed_equiv}, there exists a unique $\ol{\theta}\in(\theta,\theta+\tau)$ such that $\mu_{\tau}=v_{\ol{\theta}}$ and we may write
$$
U_{\tau}(z)=\integ{y}{0}[\infty]<S_{\theta,\tau}(U_{\tau}(y))V_{\ol{\theta}}'(z-y)>.
$$
Hence, for all $\xi \in (0,\tau],$
\begin{align*}
\xi&=U_{\tau}(\Delta_{\tau}(\xi))-U_{\tau}(0) \\
&=\integ{y}{0}[\infty]<S_{\theta,\tau}(U_{\tau}(y))\left[V_{\ol{\theta}}'(\Delta_{\tau}(\xi)-y)-V_{\ol{\theta}}'(-y)\right]> \\
&= \integ{y}{0}[\Delta_{\tau}(\tau)]<\left[S_{\theta,\tau}(U_{\tau}(y))\right]'\left[V_{\ol{\theta}}(\Delta_{\tau}(\xi)-y)-V_{\ol{\theta}}(-y)\right]> \numberthis \label{eq: inv_bound_main}
\end{align*}
after integrating by parts. Since $V_{\ol{\theta}}(z)<V_{\ol{\theta}}(0)$ for $z\in(-\infty,0)$ and $V_{\ol{\theta}}(z)>V_{\ol{\theta}}(0)$ for $z\in(0,\infty),$ we see that all integrand terms in \eqref{eq: inv_bound_main} are strictly positive. As a result, we may conclude 
\begin{align*}
\xi &> \left[\integ{y}{0}[\Delta_{\tau}(\tau)]<\left[S_{\theta,\tau}(U_{\tau}(y))\right]'>\right]\left[\min_{y\in [0,\Delta_{\tau}(\xi)]}(V_{\ol{\theta}}(\Delta_{\tau}(\xi)-y)-V_{\ol{\theta}}(-y))\right]\\
&= \min_{y\in [0,\Delta_{\tau}(\xi)]}(V_{\ol{\theta}}(\Delta_{\tau}(\xi)-y)-V_{\ol{\theta}}(-y)) \\
&> \min_{\substack{y \in [0,\Delta_{\tau}(\xi)] \\ \delta\in [\theta,\theta+\tau]}} (V_\delta(\Delta_{\tau}(\xi)-y)-V_\delta(-y)).
\end{align*}
A nearly identical argument shows 
$$
\xi< \max_{\substack{y \in [0,\Delta_{\tau}(\xi)] \\ \delta\in [\theta,\theta+\tau]}} (V_\delta(\Delta_{\tau}(\xi)-y)-V_\delta(-y)).
$$
Finally, if $\tau\in[\f{2}-\theta,1-2\theta],$ we can only assume $0\leq\mu_{\tau}<v_\theta,$ and therefore, $\ol{\theta}\in (\theta,\f{2}].$ With this adjustment, we apply the same argument, but the bounds become 
$$
\min_{\substack{y \in [0,\Delta_{\tau}(\xi)] \\ \delta\in [\theta,\f{2}]}} (V_\delta(\Delta_{\tau}(\xi)-y)-V_\delta(-y))<\xi<\max_{\substack{y \in [0,\Delta_{\tau}(\xi)] \\ \delta\in [\theta,\f{2}]}} (V_\delta(\Delta_{\tau}(\xi)-y)-V_\delta(-y)).
$$
This completes the proof of \cref{lemma: Global_Utau_Bounds}.

\section{Proof of Existence of Fronts} \label{sec: E_fronts}
Having established some global properties of the solutions, our goal in this section is to prove \cref{thm: continuation,thm: perturb}. The first part is dedicated to setting up the implicit function theorem, while the second part involves passing solutions through subsequential limits.
\subsection{Ambient Functional Spaces and Preliminaries}
In this section, our main tool is using continuous Gateaux derivatives in order to apply the implicit function theorem in Banach spaces. However, we note that the objects $(U_{\tau},\mu_{\tau},S_{\theta,\tau})$ do not perturb in the strict ``open ball" sense. In particular, since any fronts we obtain are translation invariant, it easily follows that we cannot truly use the implicit function theorem in its classical form. The workaround, first proven in \cite{ExistenceandUniqueness-ErmMcLeod} and also later used in \cite{Bates1997,ermentrout2010stimulus}, is to use properties of eigenfunctions to $L^2$ adjoint operators to fix the translation. The space that defines the translations is convex; Gateaux derivatives act on functions that are admissible in order to stay in the space.

Recall
\begin{equation} \tag{\ref{eq: mapping}}
 F[U,\mu,S_{\theta,\tau}](z):= \mu U'(z) + U(z) -\integ{y}{\mathbb{R}}[]<K(z-y)S_{\theta,\tau}(U(y))>,
\end{equation}
and a traveling wave solution solves $F\equiv 0.$ In all analysis, we will fix $\theta$ and restrict our domain to the ambient space $\mathcal{C}\times\R_+\times\{S_{\theta,\tau}\}_{\tau\geq 0},$ where $$\mathcal{C}=\{U\in C_b^2(\mathbb{R}): U(-\infty)=0,\, U(\infty)=1,\, U^{(j)}(\pm \infty)=0 \textnormal{ for } j=1,2. \}$$
is a convex subset of $C_b^2(\mathbb{R}).$ As we will see, the $U$ space will be refined further at perturbation steps. By the assumptions on $K,$ it follows that
\begin{align*}
F'[U,\mu,S_{\theta,\tau}](z)&=\mu U'' + U'-\integ{y}{\mathbb{R}}[]<K'(z-y)S_{\theta,\tau}(U(y))> \\
&=\mu U'' + U'-\integ{y}{\mathbb{R}}[]<K'(y)S_{\theta,\tau}(U(z-y))>.
\end{align*}
Note that $S_{\theta,\tau}\leq 1$ and $K'\in L^1(\mathbb{R}).$ Therefore, for $(U,\mu,S_{\theta,\tau})\in \mathcal{C}\times\R_+\times\{S_{\theta,\tau}\}_{\tau\geq 0},$ we have $F[U,\mu,S_{\theta,\tau}]\in C_0^1(\mathbb{R})$ by dominated convergence theorem, even when $\tau=0.$ Here,
$$
(C^1_0(\mathbb{R}),\norm{\cdot}_{1,\infty})=\{f\in C^1(\mathbb{R}): f(\pm \infty)=f'(\pm \infty)=0\}
$$
with norm $\norm{f}_{1,\infty}=\norm{f}_\infty + \norm{f'}_\infty.$

When $\tau\geq0$, the goal is to continuously apply the implicit function theorem to $F\equiv0$ by setting up a Newton mapping over Banach spaces. We will focus on properties of Gateaux derivatives when $\tau>0,$ since delta distributions arise when $\tau=0.$ We remark that all of the techniques are designed to hold when delta distributions are treated formally, but since $S_{\theta,0}'(u)=\delta_\theta(u)$ only formally, we perform the analysis without invoking the delta distribution.

Since Gateaux derivatives require a direction of differentiation, we define the relevant Banach space.
\begin{definition}(\cite{accinelli2010generalization})
Let $\mathcal{C}\subset X$ be a convex subset of a Banach space $X$. We say that a vector $h\in X$ is {\it admissible} for $x\in \mathcal{C}$ if and only if $x+h\in \mathcal{C}$.
\end{definition}
It is easy to see that the Banach space
$$
(C^2_0(\mathbb{R}),\norm{\cdot}_{2,\infty})=\{f\in C^2(\mathbb{R}): f(\pm \infty)=f'(\pm \infty)=f''(\pm \infty)=0\},
$$
under the norm $\norm{f}_{2,\infty}=\norm{f}_\infty + \norm{f'}_\infty + \norm{f''}_\infty$, forms the admissible directions for $\mathcal{C}.$ Taking the Gateaux derivatives with respect to $U,$ we have $DF_U: C^2_0(\mathbb{R}) \to C^1_0(\mathbb{R})$ given by
\begin{align}
DF_U[U ,\mu, S_{\theta,\tau}](h_u)(z)=\mu h_u' + h_u - \integ{y}{\mathbb{R}}[]<K(z-y)S_{\theta,\tau}'(U(y))h_u(y)>.\label{eq: Frech_U}
\end{align}
We notice that $DF_U$ can be written as $DF_U=G_2-G_1,$ where
\begin{align}
G_1[U,\mu,S_{\theta,\tau}](h_u)(z) &:=\integ{y}{\mathbb{R}}[]<K(z-y)S_{\theta,\tau}'(U(y))h_u(y)>, \label{eq: G1_op}\\
G_2[U,\mu,S_{\theta,\tau}](h_u)(z) &:= \mu h_u' + h_u . \label{eq: G2_op}
\end{align}
Note that $\mu>0$ so $G_2$ is never the identity operator. Moreover, since $U\in\mathcal{C},$ we see that in \eqref{eq: G1_op}, the domain where $S_{\theta,\tau}'(U(y))>0$ is bounded so $G_1$ is the standard Hilbert-Schmidt integral operator. Using the assumptions on $K,$ it is easy to see that $G_1$ is compact in $C_0^2(\mathbb{R})\to C_0^1(\mathbb{R}).$  The operator $G_2$ is invertible with inverse 
\begin{equation}\label{eq: G2_inv_form}
G_2^{-1}[U,\mu,S_{\theta,\tau}](v)(z)=\f{\mu}\integ{x}{-\infty}[z]<e^{\f{x-z}{\mu}}v(x)>.
\end{equation}
Combined, it follows that $G_2^{-1}G_1: C_0^2(\mathbb{R})\to C_0^2(\mathbb{R})$ is compact. Note that there are no difficulties establishing that $DF_U$ is continuous in operator norm.
\subsection{Null Space of $DF_U$ at Solutions} \label{subsec: Null}
\subsubsection*{Pointwise in $\boldsymbol{C_0^2(\mathbb{R})}$}
Let $(U_{\tau},\mu_{\tau})$ be a traveling wave solution satisfying (A2) with the translation $U_{\tau}^{-1}(\theta)=0.$ Then by differentiating $F[U_{\tau},\mu_{\tau},S_{\theta,\tau}](z)=0$ with respect to $z,$ it is easy to see that $$DF_U[U_{\tau},\mu_{\tau},S_{\theta,\tau}](h_u)(z)=0$$ when $h_u=U_{\tau}'.$ Moreover, we are able to use the assumption $\Delta_{\tau}(\xi)\leq \sigma_1(\theta)$ to establish that $\lambda=0$ is a simple eigenvalue of $DF_U[U_{\tau},\mu_{\tau},S_{\theta,\tau}]$. First, we need \cref{lemma: r_positive,cor: r_positive} below, which along with helping to prove simplicity, also play instrumental roles in highlighting the purpose of $\sigma_1(\theta),$ which is to ensure that the implicit function theorem can be applied later on.
\begin{lemma}\label{lemma: r_positive}
The function 
$$
r(\xi,\mu)= \f{\mu}\integ{x}{-\infty}[0]<e^{\f{x}{\mu}}K(x+\xi)>
$$
is positive for $\xi\in[-\sigma_1(\theta),\sigma_1(\theta)]$ and $\mu\in(0,v_\theta).$
\end{lemma}
\begin{proof}
We first prove the claim when $\xi=-\sigma_1(\theta).$ Suppose $r(-\sigma_1(\theta),\mu_*)=0$ for some $\mu_*>0.$ Then
\begin{align*}
r_\mu(-\sigma_1(\theta),\mu_*)&=\f{\mu_*^3}\integ{x}{-\infty}[0]<|x|e^{\f{x}{\mu_*}}K(x-\sigma_1(\theta))> \\
&=\f{\mu_*^3}\left(\int_{-\infty}^{\sigma_1(\theta)-M}+\int_{\sigma_1(\theta)-M}^0 \right) |x|e^{\f{x}{\mu_*}}K(x-\sigma_1(\theta)) \, \mathrm{d}x \\
&< \f{|\sigma_1(\theta)-M|}{\mu_*^3}\left(\int_{-\infty}^{\sigma_1(\theta)-M}+\int_{\sigma_1(\theta)-M}^0 \right) e^{\f{x}{\mu_*}}K(x-\sigma_1(\theta)) \, \mathrm{d}x \\
&= \f{|\sigma_1(\theta)-M|}{\mu_*^2}\underbrace{r(-\sigma_1(\theta),\mu_*)}_{=0}.
\end{align*}
Therefore, the function $r(-\sigma_1(\theta),\cdot)$ has at most one positive root. Recalling the Heaviside solutions, since $r(-\sigma_1(\theta),v_\theta)=V_\theta '(-\sigma_1(\theta))=0$ by the definition of $\sigma_1(\theta),$ this implies $v_\theta$ is the unique root. Since $r_\mu(-\sigma_1(\theta),v_\theta)<0,$ clearly $r(-\sigma_1(\theta),\mu)>0$ for $\mu \in (0,v_\theta).$ To prove the claim for $\xi \in (-\sigma_1(\theta),\sigma_1(\theta)],$ we use a change of variable and write 
\begin{align*}
r(\xi,\mu)&=\f{e^{\f{-\xi}{\mu}}}{\mu}\integ{x}{-\infty}[\xi]<e^{\f{x}{\mu}}K(x)> \\
&> \f{e^{\f{-\xi}{\mu}}}{\mu}\integ{x}{-\infty}[-\sigma_1(\theta)]<e^{\f{x}{\mu}}K(x)>,
\end{align*}
which is positive since it has the same sign as $r(-\sigma_1(\theta),\mu).$ This completes the proof.
\end{proof}
Observing that $r(\xi,v_{\ol{\theta}})=V_{\ol{\theta}}'(\xi)$, we arrive at the following corollary.
\begin{corollary}\label{cor: r_positive}
Let $(U_{\tau},\mu_{\tau})$ be a traveling wave solution. By \cref{cor: speed_equiv}, we may find the unique $\ol{\theta}\in(\theta,\f{2}]$ such that $\mu_{\tau}=v_{\ol{\theta}}.$ Then the Heaviside solution satisfies $V_{\ol{\theta}}'(\xi)>0$ for $\xi\in[-\sigma_1(\theta),\sigma_1(\theta)].$
\end{corollary}
Using a technique similar to \cite[Theorem 4.2]{ExistenceandUniqueness-ErmMcLeod}, we prove that $\lambda=0$ is simple.
\begin{lemma}\label{lemma: zero_simple_eig}
For $\tau>0,$ suppose $F[U_{\tau},\mu_{\tau},S_{\theta,\tau}]\equiv 0$ and the solution $U_{\tau}$ satisfies $\Delta_{\tau}(\xi)\leq \sigma_1(\theta)$. Then $U_{\tau}'$ is the only eigenfunction of $DF_U$ when $\lambda=0.$
\end{lemma}
\begin{proof}
We notice that $h_u$ is an eigenfunction at $\lambda=0$ if and only if $(G_2-G_1)(h_u)=0$ if and only if $(G_2^{-1}G_1-\mathcal{I})(h_u)=0$ if and only if
\begin{equation}
h_u(z)=\int_0^{\Delta_{\tau}(\tau)}S_{\theta,\tau}'(U_{\tau}(y))h_u(y)\left[\f{\mu_{\tau}} \int_{-\infty}^0 e^{\f{x}{\mu_{\tau}}}K(z+x -y)\, \mathrm{d}x\right]\mathrm{d}y. \label{equation: eig_reduced}
\end{equation}
By \cref{cor: speed_equiv}, there exists a unique $\ol{\theta}\in(\theta,\f{2})$ such that $\mu_{\tau}=v_{\ol{\theta}}$ and \eqref{equation: eig_reduced} reduces to 
$$
h_u(z)=\int_0^{\Delta_{\tau}(\tau)}S_{\theta,\tau}'(U_{\tau}(y))V_{\ol{\theta}}'(z-y)h_u(y)\, \mathrm{d}y. 
$$
Consider if there was another eigenfunction, $\zeta(z)$. Then for any constant $c$, we have
\begin{equation}\label{eq: Uprime_zeta}
U_{\tau}'(z)+c\zeta(z)=\int_0^{\Delta_{\tau}(\tau)}S_{\theta,\tau}'(U_{\tau}(y))V_{\ol{\theta}}'(z-y)(U_{\tau}'(y)+c\zeta(y)) \, \mathrm{d}y. 
\end{equation}
Consider the domain of $z\in[0,\Delta_{\tau}(\tau)]\subset [0,\sigma_1(\theta)],$ where $U_{\tau}'(z)>0.$ Then we see that $z-y\in[-\sigma_1(\theta),\sigma_1(\theta)]$ so by \cref{cor: r_positive}, the terms $S_{\theta,\tau}'(U_{\tau}(y))V_{\ol{\theta}}'(z-y)$ are positive. Since $U_{\tau}'(y)>0$ in the  range of integration, there exists $c_0$ such that $U_{\tau}'(y)+c_0\zeta(y)\geq 0$, with the set of equality being nonempty. 

Let  $z_0$ be such a point where $U_{\tau}'(z_0)+c_0\zeta(z_0)= 0$. Plugging back into \eqref{eq: Uprime_zeta} at $z=z_0$, the left hand side is zero and the right hand side consists of integrals of nonnegative functions, positive on sets of positive measure. Obviously, the right hand side is positive if $U_{\tau}'\not\equiv -c_0 \zeta$ and we arrive at a contradiction. Hence, $U_{\tau}'\equiv -c_0 \zeta$. 

\end{proof}
\subsubsection{Positive Eigenfunctions of the $L^2([0,\Delta_{\tau}(\tau)])$ Adjoint} \label{subsubsec: Adjoint}
In the proof of \cref{lemma: zero_simple_eig}, we saw that at solutions $(U_{\tau},\mu_{\tau}),$ there exists a unique $\ol{\theta}\in(\theta,\f{2})$ such that 
\begin{equation}\label{equation: G2inv_G1}
G_2^{-1}G_1(h)(z)=\int_0^{\Delta_{\tau}(\tau)}S_{\theta,\tau}'(U_{\tau}(y))V_{\ol{\theta}}'(z-y)h(y)\, \mathrm{d}y,
\end{equation}
and this operator has a simple eigenvalue at $\ol{\lambda}=1.$ Consider this same operator, but in a real Hilbert space setting from $L^2([0,\Delta_{\tau}(\tau)])\to L^2([0,\Delta_{\tau}(\tau)])$ with inner product
$$
\langle f, g \rangle = \int_0^{\Delta_{\tau}(\tau)} f(y) g(y)\, \mathrm{d}y.
$$
To avoid confusion between the action on Banach versus Hilbert spaces, call this operator $\mathcal{T_H}:L^2([0,\Delta_{\tau}(\tau)])\to L^2([0,\Delta_{\tau}(\tau)])$ given by 
$$
\mathcal{T}_H(h)(z)=\int_0^{\Delta_{\tau}(\tau)}S_{\theta,\tau}'(U_{\tau}(y))V_{\ol{\theta}}'(z-y)h(y)\, \mathrm{d}y.
$$
Firstly, we observe that $\mathcal{T_H}$ is still compact and $\mathcal{T_H}(h)(z)$ equals a pointwise continuous function almost everywhere. Secondly, since $C^2([0,\Delta_{\tau}(\tau)])\subset  L^2([0,\Delta_{\tau}(\tau)]),$ it also follows that $h=U_{\tau}'$ is still an eigenfunction of $\mathcal{T_H}$ associated with $\ol{\lambda}=1.$ Since $\mathcal{T_H}(h)(z)$ is continuous, we see that all solutions to $h(z)=\mathcal{T_H}(h)(z)$ have a continuous representation. By \cref{lemma: zero_simple_eig}, the only such continuous solution is $h=U_{\tau}'$. Therefore, $\ol{\lambda}=1$ is a simple eigenvalue of $\mathcal{T_H}$ so simplicity is preserved in $L^2([0,\Delta_{\tau}(\tau)])$.


We now use a very similar method to \cite[Theorem 4.3]{ExistenceandUniqueness-ErmMcLeod} in order to prove that the solution to the $L^2$ adjoint equation is nonnegative almost everywhere. A standard calculation of the Hilbert-Schmidt integral operator shows the adjoint is given by
\begin{align} \label{equation: Hilbert_T_Adjoint}
\mathcal{T}_H^*(\psi)(z)=S_{\theta,\tau}'(U_{\tau}(z))\integ{y}{0}[\Delta_{\tau}(\tau)]<V_{\ol{\theta}}'(y-z)\psi(y)>.
\end{align}
\begin{lemma} \label{lemma: psi_tau_star_positive}
There exists exactly one solution $\ol{\psi}_{\tau}\in L^2([0,\Delta_{\tau}(\tau)])$ to $(\mathcal{T}_H^*-\mathcal{I})(\psi)=0.$ The solution is of one sign almost everywhere.
\end{lemma}
\begin{proof}
Since $\mathcal{T}_H$ is compact and $h=U_{\tau}'$ is the only solution to $(\mathcal{T}_H-\mathcal{I})(h)=0$, the existence of exactly one solution, $\ol{\psi}_{\tau},$ to $(\mathcal{T}_H^*-\mathcal{I})(\psi)=0,$ follows from the Fredholm alternative. 

We now show $\ol{\psi}_{\tau}$ is of one sign almost everywhere. Take any function from the $L^2$ equivalence class of functions such that almost everywhere,
$$
\ol{\psi}_{\tau}(z)=S_{\theta,\tau}'(U_{\tau}(z))\integ{y}{0}[\Delta_{\tau}(\tau)]<V_{\ol{\theta}}'(y-z)\ol{\psi}_{\tau}(y)>.
$$
Decompose $\ol{\psi}_{\tau}$ into positive and negative parts as $\ol{\psi}_{\tau}=\ol{\psi}_+-\ol{\psi}_-.$ Much like in \cite[Theorem 4.3]{ExistenceandUniqueness-ErmMcLeod}, we know that by \cref{cor: r_positive}, since $V_{\ol{\theta}}'(y-z)>0$ for $z\in[0,\Delta_{\tau}(\tau)]\subset [0,\sigma_1(\theta)],$ it easily follows that $(\mathcal{T}_H^*-\mathcal{I})(\ol{\psi}_+)\geq 0$ and $(\mathcal{T}_H^*-\mathcal{I})(-\ol{\psi}_-)\leq 0$ almost everywhere. Using the definition of adjoint,
\begin{equation} \label{eq: innerprod_pos}
0=\langle (\mathcal{T}_H-\mathcal{I})(U_{\tau}'),\psi\rangle =\langle U_{\tau}',(\mathcal{T}_H^*-\mathcal{I})(\psi)\rangle
\end{equation}
for any $\psi.$ Separately plugging in $\ol{\psi}_+$ or $\ol{\psi}_-$ into \eqref{eq: innerprod_pos}, we find that $\ol{\psi}_+ \equiv 0$ or $\ol{\psi}_- \equiv 0.$
\end{proof}
As we will see in our Newton mapping setup, our motivation for showing $\ol{\psi}_{\tau}$ to be (without loss of generality) nonnegative almost everywhere is that we need to guarantee that two important $L^2$ inner products are nonzero. The following lemma proves this fact.
\begin{lemma} \label{lemma: adjoint_Up_notzero}
For $\tau>0,$ let $(U_{\tau},\mu_{\tau})$ be a solution satisfying (A2) with $\Delta_{\tau}(\xi)\leq \sigma_1(\theta).$ Without loss of generality, let $\ol{\psi}_{\tau}\geq 0$ almost everywhere. Then
\begin{itemize}
\item[(i)] $\langle U_{\tau}',\ol{\psi}_{\tau}\rangle > 0,$
\item[(ii)] $\langle G_2^{-1}[U_{\tau},\mu_{\tau},S_{\theta,\tau}](U_{\tau}'),\ol{\psi}_{\tau}\rangle > 0,$
\end{itemize}
where inner products are with respect to $L^2([0,\Delta_{\tau}(\tau)]).$
\end{lemma}
\begin{proof}
The proof of (i) is trivial since both terms in the inner product are nonnegative. For (ii), it suffices to show $G_2^{-1}[U_{\tau},\mu_{\tau},S_{\theta,\tau}](U_{\tau}')(z)>0$ for $z\in[0,\Delta_{\tau}(\tau)].$ By definition,
\begin{align*}
G_2^{-1}[U_{\tau},\mu_{\tau},S_{\theta,\tau}](U_{\tau}')(z)&=\f{\mu_{\tau}}\integ{x}{-\infty}[z]<e^{\f{x-z}{\mu_{\tau}}}U_{\tau}'(x)> \\
&= \f{\mu_{\tau}}\left[U_{\tau}(z)-\f{\mu_{\tau}} \integ{x}{-\infty}[z]<e^{\f{x-z}{\mu_{\tau}}}U_{\tau}(x)> \right]
\end{align*}
after integrating by parts. Since $U_{\tau}(\cdot)<\theta$ on $(-\infty,0)$ and $U_{\tau}'(\cdot)>0$ on $[0,\Delta_{\tau}(\tau)],$ it must be true that for $z\in[0,\Delta_{\tau}(\tau)]$ fixed, we have $U_{\tau}(x)<U_{\tau}(z)$ on $(-\infty,z).$ Since exponentials preserve inequalities, it follows that above is greater than 
\begin{align*}
\f{\mu_{\tau}}\left[U_{\tau}(z)-\f{U_{\tau}(z)}{\mu_{\tau}}\integ{x}{-\infty}[z]<e^{\f{x-z}{\mu_{\tau}}}> \right]=0.
\end{align*}
\end{proof}
\subsection{Newton's Method in Banach Spaces}
With the exception of $\tau=0$ subtleties, we are in a position to show how to cleanly apply the implicit function theorem in Banach spaces to $F[U_{\tau},\mu_{\tau},S_{\theta,\tau}]=0$ in order to generate more solutions in the interval $[\tau,\ol{\tau}).$  Throughout this section, assume $U_{\tau}^{-1}(\theta)=0$ and $\Delta_{\tau}(\xi)\leq \sigma_1(\theta).$ Recall that our ambient $U$ space is the convex set
$$\mathcal{C}=\{U\in C_b^2(\mathbb{R}): U(-\infty)=0,\, U(\infty)=1,\, U^{(j)}(\pm \infty)=0 \textnormal{ for } j=1,2. \}.$$
We also note that any collection of firing rates $\{S_{\theta,\tau} \}_{\tau\geq 0}$ continuously deformed in $L^1$ by $\tau,$ and defined by \eqref{eq: smooth_heav}, forms a subset of $L^1(\mathcal{T}),$ where $\mathcal{T} \supset [0,1]$ is a large interval that contains all possible values of $U$ of interest.

Importantly, deforming the firing rates in $L^1$ is necessary. This is especially critical at the $\tau=0$ step since $\norm{S_{\theta,\tau}-S_{\theta,0}}_1 \to 0$ as $\tau\to 0,$ while this is untrue in the $L^\infty$ norm due to the Heaviside function having a jump. Indeed, using the $L^1$ norm is how we overcome the difficulty of the delta distribution arising when $\tau=0.$
\subsubsection*{When $\tau>0$}
At our $\tau>0$ perturbations, we fix the translation invariance problem by reducing $\mathcal{C}$ to the convex subset
\begin{align}
E(U_{\tau}, \mu_{\tau},\ol{\psi}_{\tau})=\{U\in \mathcal{C}: \langle U-U_{\tau},\ol{\psi}_{\tau}\rangle=0\},\label{eq: E_tau}
\end{align}
where the inner product is over $L^2([0,\Delta_{\tau}(\tau)]).$ 
Here $\ol{\psi}_{\tau}$ is the solution to $\mathcal{T}_H^*(\psi)=\psi,$ with all properties as described in \cref{subsubsec: Adjoint}.
The space of admissible directions,
\begin{align}
A(U_{\tau},\mu_{\tau},\ol{\psi}_{\tau}) &:=\{ U_1-U_2: U_1,\, U_2 \in E(U_{\tau}, \mu_{\tau},\ol{\psi}_{\tau})\}, \label{eq: admissible_Atau} 
\end{align}
forms a closed subspace of the Banach space $(C^2_0(\mathbb{R}),\norm{\cdot}_{2,\infty}).$ This setup is similar to that in \cite{ExistenceandUniqueness-ErmMcLeod}, but with different assumptions on the firing rate and kernel.
\subsubsection*{When $\tau=0$}
In principle, the spaces when $\tau=0$ are derived similarly to those when $\tau>0.$ However, in \cref{subsec: Null}, due to the fact that $S_{\theta,0}'=\delta_\theta$ in distribution, the formal representation of the Gateaux derivative at solutions $(U_{0},\mu_{0},S_{\theta,0})=(V_\theta,v_\theta,H(u-\theta))$ takes on the form
\begin{equation}
G_2^{-1}G_1(h)(z)=\f{V_\theta'(z)h(0)}{V_\theta'(0)}.
\end{equation}
This linearization is well-known and widely used to derive the Evan's function \cite{CoombesOwen_Evans, Zhang-OnStability} in order to study the stability of Heaviside solutions. However, in a technical sense, our techniques break down since this mapping is not well-defined on $L^2.$  Moreover, by using the form of \eqref{equation: Hilbert_T_Adjoint}, the $L^2$ adjoint operator $\mathcal{T}_H^*$ formally becomes
$$
\mathcal{T}_H^*(\psi)(z)=\delta(V_\theta(z)-\theta)\integ{y}{\mathbb{R}}[]<V_\theta(y-z)\psi(y)>.
$$
On intuition alone, to arrange for $U\in \mathcal{C}$ such that $\langle U-V_\theta,\ol{\psi}_{\theta,0} \rangle = 0,$ it seems appropriate to demand $U(0)=V_\theta(0)=\theta,$ which is precisely what we do. Hence, we reduce to the convex subset
\begin{equation}
E_\theta=\{U\in \mathcal{C}: U(0)=V_\theta(0)=\theta\}.
\end{equation}
The Banach space of admissible directions is then
\begin{equation}
A_\theta=\{U_1-U_2: U_1,\, U_2 \in E_\theta\}=\{f\in C_0^2(\mathbb{R}): f(0)=0\}.
\end{equation}
\subsubsection*{Continuous Extension of Gateaux Derivatives When $\tau=0$}
In order to correctly apply the implicit function theorem on convex sets, we must show near $(U_{\tau},\mu_{\tau},S_{\theta,\tau})$, the partial Gateaux derivatives in all admissible directions are continuous in operator norm and, since $F[U,\mu,S_{\theta,\tau}]\in C_0^1(\mathbb{R})$ for $U\in\mathcal{C},$ the mapping $DF_{U,\mu}[U_{\tau},\mu_{\tau},S_{\theta,\tau}]: A(U_{\tau},\mu_{\tau}) \times \mathbb{R} \to C^1_0(\mathbb{R})$ is a Banach space isomorphism. 

Since $S_{\theta,\tau}$ is smooth for $\tau>0$, the continuity of Gateaux derivatives is trivial in this case. Moreover, taking the Gateaux derivative with respect to $\mu,$ we see that $$DF_\mu[U,\mu,S_{\theta,\tau}](h_\mu)(z)=h_\mu U'(z),$$ which does not cause any issues involving the delta distribution. Finally, the mapping
$$
(U,\mu,S_{\theta,\tau})\mapsto F[U,\mu,S_{\theta,\tau}],
$$
with respect to the norms $\norm{\cdot}_{2,\infty} + |\cdot| + \norm{\cdot}_1\to \norm{\cdot}_{1,\infty}$, poses no continuity issues since the delta distribution does not arise.

The only continuity issue that we have to deal with is the mapping 
$$
(U,\mu,S_{\theta,\tau})\mapsto DF_U[U,\mu,S_{\theta,\tau}]
$$
for $(U,\mu,S_{\theta,\tau})\in E_\theta \times \mathbb{R}_+\times \{S_{\theta,\tau}\}_{\tau\geq 0}$ in a neighborhood (with respect to admissible directions) of $(U_{0},\mu_{0},S_{\theta,0})=(V_\theta,v_\theta,H(u-\theta)).$ In particular, we need to continuously extend the mapping to the $\tau=0$ case. Formally, when $\tau=0,$ the Gateaux derivative for $U\in E_\theta$ in the direction of $A_\theta$ is given by 
\begin{equation}
DF_U[U,\mu,S_{\theta,0}](h_u)(z) = \mu h_u' + h_u -\f{K(z)h_u(0)}{U'(0)}.
\end{equation}
Note that since $U$ is in a neighborhood of $V_\theta$ in the $\norm{\cdot}_{2,\infty}$ norm, and $V_\theta(0)=\theta,$ $V_\theta' > 0$ when $V_\theta \in [0,1],$ there are no issues with assuming $U'(0)>0$ since $U\in E_\theta$ guarantees that $U(0)=V_\theta(0)=\theta.$ Moreover, since $h_u\in A_\theta,$ it follows that $h_u(0)=0$ so the last term vanishes. 

With this observation in mind, define the following extension of the Gateaux derivative, $\ol{DF}_U: E_\theta \times \mathbb{R}_+\times \{S_{\theta,\tau}\}_{\tau\geq 0}\to \mathcal{B}(A_\theta,C_0^1(\mathbb{R}))$, where the image is the space of bounded linear transformations from the Banach space $A_\theta$ to the Banach space $C_0^1(\mathbb{R}).$
\begin{equation}
\ol{DF}_U[U,\mu,S_{\theta,\tau}](h_u)=\begin{cases}
DF_U[U,\mu,S_{\theta,\tau}](h_u),& \tau>0 \\
\mu h_u' + h_u,& \tau=0
\end{cases}
\end{equation}
Similarly, we will use the notation $\ol{DF}_{U,\mu}[U,\mu,S_{\theta,\tau}]=\ol{DF}_{U}[U,\mu,S_{\theta,\tau}]+ DF_{\mu}[U,\mu,S_{\theta,\tau}].$
\begin{lemma}\label{lemma: Gat_cont}
When $\tau\geq 0,$ the mapping $\ol{DF}_U$ is continuous at every point with respect to operator norm.
\end{lemma}
\begin{proof}
The proof is trivial when the point $(U_*,\mu_*,S_{\theta,\tau})$ satisfies $\tau>0.$ Hence, assume the point is of the form $(U_*,\mu_*,S_{\theta,0}).$ Let $h_u\in A_\theta$ with $\norm{h_u}_{2,\infty}=1.$ For all nearby $(U,\mu,S_{\theta,\tau})$ and all $z,$ 
\begin{align*}
&|\ol{DF}_U[U,\mu,S_{\theta,\tau}](h_u)(z)-\ol{DF}_U[U_*,\mu_*,S_{\theta,0}](h_u)(z)| \\
&\leq |\ol{DF}_U[U,\mu,S_{\theta,\tau}](h_u)(z)-\ol{DF}_U[U_*,\mu,S_{\theta,\tau}](h_u)(z)| \numberthis \label{eq: A.1} \\
&+ |\ol{DF}_U[U_*,\mu,S_{\theta,\tau}](h_u)(z)-\ol{DF}_U[U_*,\mu_*,S_{\theta,\tau}](h_u)(z)| \numberthis \label{eq: A.2}\\
&+|\ol{DF}_U[U_*,\mu_*,S_{\theta,\tau}](h_u)(z)-\ol{DF}_U[U_*,\mu_*,S_{\theta,0}](h_u)(z)| \numberthis \label{eq: A.3}
\end{align*}
If $\tau>0,$ term \eqref{eq: A.1} reduces to 
\begin{align*}
&\left|\integ{y}{\mathbb{R}}[]<K(z-y)(S_{\theta,\tau}'(U(y))-S_{\theta,\tau}'(U_*(y)))h_u(y)>\right| \\
&\leq \max |S_{\theta,\tau}''(\cdot)|\norm{K}_1  \norm{h}_{2,\infty} \norm{U-U_*}_{2,\infty} \longrightarrow 0
\end{align*}
as $\norm{U-U_*}_{2,\infty} \longrightarrow 0$ independent of $h$ and $z$. If $\tau=0,$ then term \eqref{eq: A.1} reduces to 
$$
|(\mu h_u' + h_u)-(\mu h_u' + h_u)|=0.
$$
In \eqref{eq: A.2}, for $\tau\geq 0,$ the term reduces to
$$
|h_u'(\mu-\mu_*)| \leq |\mu-\mu_*| \longrightarrow 0
$$
as $|\mu-\mu_*|\longrightarrow 0,$ independent of $h_u.$ Finally, term \eqref{eq: A.3} vanishes if $\tau=0$ and if $\tau>0,$ it reduces to
\begin{align} \label{eq: Third_cont_term}
\left|\integ{y}{\mathbb{R}}[]<K(z-y)S_{\theta,\tau}'(U_*(y))h_u(y)>\right|.
\end{align}
Since $U_*\in E_\theta$, the quantities $\norm{U_*-V_\theta}_{2,\infty}$ and $\tau>0$ are small, and the shape of $V_\theta$ guarantees that $U_* ' >0$ when $U_* \in [\theta,\theta+\tau].$ Therefore, we may assume that $U_*$ is invertible (with positive inverse) when $U_* \in [\theta,\theta+\tau].$ Hence, we may rewrite \eqref{eq: Third_cont_term} as 
\begin{align*}
&\phantom{=}\left|\integ{\xi}{0}[\tau]<\f{K(z-U_*^{-1}(\theta+\xi))h_u(U_*^{-1}(\theta+\xi))}{U_* '(U_*^{-1}(\theta+\xi))} S_{\theta,\tau}'(\theta+\xi)> \right| \\
&\stackrel{\textnormal{IBP}}{=}\left|\left[\f{K(z-U_*^{-1}(\theta+\tau))h_u(U_*^{-1}(\theta+\tau))}{U_* '(U_*^{-1}(\theta+\tau))}\right]\right. \\
&-\left.\integ{\xi}{0}[\tau]<\left[\f{K(z-U_*^{-1}(\theta+\xi))h_u(U_*^{-1}(\theta+\xi))}{U_* '(U_*^{-1}(\theta+\xi))}\right]' S_{\theta,\tau}(\theta+\xi)>\right|, \\
\intertext{and since $h_u(U_*^{-1}(\theta))=h_u(0)=0$, the previous term can be written as}
&\phantom{=}\left|\integ{\xi}{0}[\tau]<\left[\f{K(z-U_*^{-1}(\theta+\xi))h_u(U_*^{-1}(\theta+\xi))}{U_* '(U_*^{-1}(\theta+\xi))}\right]'H(\xi)>\right. \\
&-\left.\integ{\xi}{0}[\tau]<\left[\f{K(z-U_*^{-1}(\theta+\xi))h_u(U_*^{-1}(\theta+\xi))}{U_* '(U_*^{-1}(\theta+\xi))}\right]' S_{\theta,\tau}(\theta+\xi)>\right| \\
&= \left|\integ{\xi}{0}[\tau]<\left[\f{K(z-U_*^{-1}(\theta+\xi))h_u(U_*^{-1}(\theta+\xi))}{U_* '(U_*^{-1}(\theta+\xi))}\right]' (\underbrace{S_{\theta,0}(\theta+\xi)}_{=H(\xi)}-S_{\theta,\tau}(\theta+\xi))>\right| \\
&\leq C \norm{S_{\theta,\tau}-S_{\theta,0}}_{L^1([\theta,\theta+\tau])} \longrightarrow 0
\end{align*}
as $\tau\to 0^+$. The constant $C>0$ can be chosen independent of $z$ and the choice of $h_u$ such that $\norm{h_u}_{2,\infty}=1.$ Moreover, the smallness of $\norm{U_*-V_\theta}_{2,\infty}$ guarantees that $U_*'(U_*^{-1}(\theta+\xi))$ is bounded below so $C$ does not blow up.

Combined, we have shown that $\norm{\ol{DF}_U[U,\mu,S_{\theta,\tau}]-\ol{DF}_U[U_*,\mu_*,S_{\theta,0}]} \longrightarrow 0$ in operator norm as $\norm{U-U_*}_{2,\infty} + |\mu-\mu_*| + \norm{S_{\theta,\tau}-S_{\theta,0}}_1 \longrightarrow 0.$ Hence, $\ol{DF}_U$ is a continuous extension of $DF_U.$
\end{proof}
In conclusion, all Gateaux derivatives are continuous in all relevant functional spaces.
\subsubsection*{Invertibility of Gateaux Derivatives at Solutions}
Following the standard proof of the implicit function theorem using Newton's method, we must show that at solutions satisfying (A2) with $\Delta_{\tau}(\xi)\leq \sigma(\theta),$ the bounded linear operator $DF_{U,\mu}^{-1}[U_{\tau},\mu_{\tau},S_{\theta,\tau}]:C_0^1(\mathbb{R})\to A(U_{\tau},\mu_{\tau},\ol{\psi}_{\tau})\times \mathbb{R}$ exists. Since $DF_{U,\mu}[U_{\tau},\mu_{\tau},S_{\theta,\tau}]: A(U_{\tau},\mu_{\tau},\ol{\psi}_{\tau})\times \mathbb{R}\to C_0^1(\mathbb{R})$ is a mapping between Banach spaces, we show that this forms a Banach space isomorphism. Then by the open mapping theorem, a bounded inverse exists.

Related to the technique in \cite[Theorem 4.4]{ExistenceandUniqueness-ErmMcLeod}, the following lemma precisely defines our Banach space isomorphisms at each perturbation.
\begin{lemma}\label{lemma: invertible_linmap}
At solutions $(U_{\tau},\mu_{\tau},S_{\theta,\tau})$ satisfying assumption (A2) with $\Delta_{\tau}(\xi)\leq \sigma(\theta),$ the Gateaux derivative $\ol{DF}_{U,\mu}[U_{\tau},\mu_{\tau},S_{\theta,\tau}]: A(U_{\tau},\mu_{\tau},\ol{\psi}_{\tau}) \times \mathbb{R} \to C_0^1(\mathbb{R})$ forms a Banach space isomorphism.
\end{lemma}
\begin{proof}
Let $g\in C_0^1(\mathbb{R}).$ We start with the (unique) $\tau=0$ case, where $(U_{0},\mu_{0},S_{\theta,0})=(V_\theta,v_\theta,H(u-\theta)).$  We wish to solve for $(h_u(z),h_\mu)$ in the equation 
$$
\ol{DF}_U[U_{0},\mu_{0},S_{\theta,0}](h_u)(z) + DF_\mu[U_{0},\mu_{0},S_{\theta,0}](h_\mu)(z) = g(z),$$
or equivalently,
\begin{equation}
\mu_{0}h_u'(z) + h_u(z) + h_\mu U_{0}'(z) =g(z).
\end{equation}
In terms of our bounded linear operators, this can be written as 
$$
G_2[U_{0},\mu_{0},S_{\theta,0}](h_u)(z)=g(z)-h_\mu U_{0}'(z).
$$
Since $G_2:C_0^2(\mathbb{R})\to C_0^1(\mathbb{R})$ is invertible (see \eqref{eq: G2_inv_form}), this is equivalent to 
\begin{equation}\label{eq: G2_inv_iso}
h_u(z)=G_2^{-1}(g-h_\mu U_{0}')(z).
\end{equation}
The condition $h_u\in A_\theta$ forces $h_u(0)=0.$ Hence, plugging in $z=0,$ we have the formula
\begin{equation}\label{eq: mu_form_iso}
h_\mu=\f{G_2^{-1}(g)(0)}{G_2^{-1}(U_{0}')(0)},
\end{equation}
where the denominator is positive by the proof of \cref{lemma: adjoint_Up_notzero} (ii). Plugging in \eqref{eq: mu_form_iso} into \eqref{eq: G2_inv_iso}, we have a unique solution $h_u \in A_\theta$ since $G_2$ is invertible.

At solutions when $\tau>0,$ the analogous equation becomes
\begin{equation}\label{eq: iso_Fredholm}
(G_2^{-1}G_1-\mathcal{I})(h_u)(z)=-G_2^{-1}(g-h_\mu U_{\tau}').
\end{equation}
Let $\ol{\psi}_{\tau}$ be the adjoint solution chosen for the space $A(U_{\tau},\mu_{\tau},\ol{\psi}_{\tau}).$ Then since $G_2^{-1}G_1$ is just $\mathcal{T}_H$, but applied to $C_0^2(\mathbb{R})$ functions, we take the $L^2([0,\Delta_{\tau}(\tau)])$ inner product of both sides with $\ol{\psi}_{\tau},$ and by definition of adjoint, we have the formula for $h_\mu$ given by
\begin{equation}
h_\mu=\f{\langle G_2^{-1}(g),\ol{\psi}_{\tau} \rangle}{\langle G_2^{-1}(U_{\tau}'),\ol{\psi}_{\tau} \rangle}.
\end{equation}
The denominator is positive by \cref{lemma: adjoint_Up_notzero} (ii). With this choice of $h_\mu,$ we are now in a position to apply the Fredholm alternative in \eqref{eq: iso_Fredholm}. As was shown in 
\cref{lemma: zero_simple_eig}, the only solution to 
\begin{equation}\label{eq: Fred_null}
(G_2^{-1}G_1-\mathcal{I})(h_u)(z)=0
\end{equation}
is $h_u=cU_{\tau}'.$ By \cref{lemma: adjoint_Up_notzero} (i), the only way $cU_{\tau}' \in A(U_{\tau},\mu_{\tau},\ol{\psi}_{\tau})$ can occur is if $c=0.$ Hence, there are no nontrivial solutions to \eqref{eq: Fred_null}. Since $G_2^{-1}G_1: A(U_{\tau},\mu_{\tau},\ol{\psi}_{\tau})\to A(U_{\tau},\mu_{\tau},\ol{\psi}_{\tau})$ is compact, the Fredholm alternatives implies \eqref{eq: iso_Fredholm} has a unique solution $h_u.$

For $\tau\geq 0,$ we have established bijections between Banach spaces. By applying the open mapping theorem, we see that these are Banach space isomorphisms in that bounded inverse operators exists.
\end{proof}
Now that we have carefully defined our spaces, we may easily apply the implicit function theorem.
\begin{lemma}[Implicit Function Theorem]\label{lemma: Implicit_smooth}
For fixed $\tau\geq 0,$ let $(U_{\tau},\mu_{\tau},S_{\theta,\tau})$ be a solution satisfying assumption (A2) with $\Delta_{\tau}(\xi)\leq \sigma(\theta).$ Then there exists $\epsilon_0>0$ and $\delta_0>0$ sufficiently small such that for $(\ol{U},\ol{\mu},S_{\theta,\ol{\tau}})\in E(U_{\tau},\mu_{\tau},\ol{\psi}_{\tau}) \times \mathbb{R}_+ \times \{S_{\theta,\ol{\tau}}\}_{\ol{\tau}\geq \tau}$ with $\norm{\ol{U}-U_{\tau}}_{2,\infty} + |\ol{\mu}-\mu_{\tau}| <\epsilon_0$ and $\norm{S_{\theta,\ol{\tau}}-S_{\theta,\tau}}_1<\delta_0,$ the mapping $\mathcal{N}: E(U_{\tau},\mu_{\tau},\ol{\psi}_{\tau}) \times \mathbb{R}_+ \times \{S_{\theta,\ol{\tau}}\}_{\ol{\tau}\geq \tau} \to E(U_{\tau},\mu_{\tau},\ol{\psi}_{\tau}) \times \mathbb{R}_+$ defined by
\begin{equation}
\mathcal{N}[\ol{U},\ol{\mu},S_{\theta,\ol{\tau}}]:=(\ol{U}, \ol{\mu}) - DF_{U,\mu}^{-1}[U_{\tau},\mu_{\tau},S_{\theta,\tau}](F[\ol{U},\ol{\mu},S_{\theta,\ol{\tau}}])
\end{equation}
has a unique fixed point $\mathcal{N}(U_{\ol{\tau}},\mu_{\ol{\tau}},S_{\theta,\ol{\tau}})=(U_{\ol{\tau}},\mu_{\ol{\tau}})$. In particular, we may write $$F[U(S_{\theta,\ol{\tau}}),\mu(S_{\theta,\ol{\tau}}),S_{\theta,\ol{\tau}}]=0$$ for $\norm{S_{\theta,\ol{\tau}}-S_{\theta,\tau}}_1<\delta_0,$ with $(U(S_{\theta,\ol{\tau}}),\mu(S_{\theta,\ol{\tau}}))$ continuous in the norm $\norm{\cdot}_{2,\infty} + |\cdot|$ with respect to changes in $S_{\theta,\ol{\tau}}$ in the norm $\norm{\cdot}_1.$
\end{lemma}
\begin{proof}
The proof is a standard application of Banach fixed-point theorem. See \nameref{sec: app_cont}.
\end{proof}
Note that in \cref{lemma: Implicit_smooth}, perturbations can be chosen to be small enough so that all solutions $(U_{\ol{\tau}},\mu_{\ol{\tau}},S_{\theta,\ol{\tau}})$ satisfy (A2). This completes the proof of the first part of \cref{thm: continuation} regarding perturbations at solutions.
\subsection{Proof of Unique Heaviside Perturbation}
Having established the application of the implicit function theorem, we finish up the proof of \cref{thm: perturb}. By \cref{lemma: Implicit_smooth}, it follows that for $0<\delta_0 \ll 1$, equation \eqref{eq: front_Utau} is satisfied for a pair $(U_{\tau},\mu_{\tau})$, with $U_{\tau}$ unique in $E_\theta$ (where $U_{\tau}(0)=V_\theta(0)=\theta$) for all $\tau< \delta_0$. We finish the proof of \cref{thm: perturb} regarding uniqueness and smoothness in the proceeding lemmas.
\begin{lemma}\label{lemma: Unique}
For $\tau \ll 1,$ all solutions satisfying (A2) are unique (modulo translation).
\end{lemma}
\begin{proof}
First of all, the base point $(V_\theta,v_\theta)$ is unique. By \cref{lemma: Implicit_smooth}, there exists $\epsilon_0>0$ and $\delta_0 >0$ such that all solutions $(U_{\ol{\tau}},\mu_{\ol{\tau}},S_{\theta,\ol{\tau}})\in E_\theta\times \mathbb{R}_+ \times \{S_{\theta,\ol{\tau}}\}_{\ol{\tau}\geq 0}$ with $\norm{U_{\ol{\tau}}-V_\theta}_{2,\infty} + |\mu_{\ol{\tau}}-v_\theta|<\epsilon_0$ and $\norm{S_{\theta,\ol{\tau}}-S_{\theta,0}}_1<\delta_0$ are unique. We show that any arbitrary solution can be translated to fit these requirements. 

Let $(U_{\ol{\tau}},\mu_{\ol{\tau}},S_{\theta,\ol{\tau}})$ be an arbitrary solution satisfying (A2).  By translation invariance, we can easily translate so that $U_{\ol{\tau}}(0)=V_\theta(0)=\theta$, and therefore, $U_{\ol{\tau}}\in E_\theta$ holds. We claim that for all such translations, we may choose $\tau_0>0$ such that for $\ol{\tau}<\tau_0,$ we have $\norm{U_{\ol{\tau}}-V_\theta}_{2,\infty} + |\mu_{\ol{\tau}}-v_\theta|<\epsilon_0$ and $\norm{S_{\theta,\ol{\tau}}-S_{\theta,0}}_1<\delta_0,$ with the choice of $\tau_0$ not depending on the particular solution. Then the uniqueness result in \cref{lemma: Implicit_smooth} completes the proof.

For a given solution, by \cref{lemma: Global_Wave_Bounds}, we have the global bounds $v_{\theta+\ol{\tau}}<\mu_{\ol{\tau}}<v_\theta$. Since $v_{\theta+\ol{\tau}}\to v_\theta$ as $\ol{\tau} \to 0,$ clearly $\mu_{\ol{\tau}} \to v_\theta$ as $\ol{\tau} \to 0$ with the convergence being uniform over all solutions.

Under the translation $U_{\ol{\tau}}^{-1}(\theta)=0,$ by \cref{lemma: Global_Utau_Bounds}, we have the global bound $\omega_L(\Delta_{\ol{\tau}}(\ol{\tau}),\ol{\tau})<\ol{\tau}.$ By the squeeze theorem, as $\ol{\tau}\to 0,$ this implies
$$
\min_{y \in [0,\Delta_{\ol{\tau}}(\ol{\tau})]} (V_\theta(\Delta_{\ol{\tau}}(\ol{\tau})-y)-V_\theta(-y)) 
$$
converges to zero. Given that $V_\theta(z)<\theta$ on $(-\infty,0)$ and $V_\theta(z)>\theta$ on $(0,\infty),$ this is only possible if $\Delta_{\ol{\tau}}(\ol{\tau}) \to 0.$ Hence, we may further restrict $\ol{\tau}$ to make $\Delta_{\ol{\tau}}(\ol{\tau})$ as small as we desire, uniform in the solution choice since the convergence rate is controlled by $V_\theta$ and all other nearby Heaviside solutions.

By \cref{cor: speed_equiv}, there exists a unique $\ol{\theta}\in(\theta,\theta+\ol{\tau})$ such that $\mu_{\ol{\tau}}=v_{\ol{\theta}}$ and we may write
\begin{align*}
U_{\ol{\tau}}(z)=\integ{y}{0}[\infty]<S_{\theta,\ol{\tau}}(U_{\ol{\tau}}(y))V_{\ol{\theta}}'(z-y)>. 
\end{align*}
Since all solutions, by translation choice, satisfy $U_{\ol{\tau}}(y)<\theta$ on $(-\infty,0)$ and $U_{\ol{\tau}}(y)>\theta$ on $(0,\infty),$ we see that $S_{\theta,0}(U_{\ol{\tau}}(y))=H(y).$ Hence, 
$$
V_{\ol{\theta}}(z)=\integ{y}{0}[\infty]<S_{\theta,0}(U_{\ol{\tau}}(y)) V_{\ol{\theta}}'(z-y)>.
$$
Subtracting, 
\begin{align*}
|U_{\ol{\tau}}(z)-V_{\ol{\theta}}(z)| &= \left|\integ{y}{0}[\infty]<\left[S_{\theta,\ol{\tau}}(U_{\ol{\tau}}(y))-S_{\theta,0}(U_{\ol{\tau}}(y))\right]V_{\ol{\theta}}'(z-y)> \right| \\
&=\left|\integ{y}{0}[\Delta_{\ol{\tau}}(\ol{\tau})]<\left[S_{\theta,\ol{\tau}}(U_{\ol{\tau}}(y))-S_{\theta,0}(U_{\ol{\tau}}(y))\right]V_{\ol{\theta}}'(z-y)> \right| 
\end{align*}
As $\ol{\tau} \to 0,$ the integrand is uniformly bounded in $z$ and solution choice, while $\Delta_{\ol{\tau}}(\ol{\tau}) \to 0$ uniformly. Hence, $\norm{U_{\ol{\tau}}-V_\theta}_\infty \longrightarrow 0$ uniformly in solution choice. Finally, from the coupling equations
\begin{align*}
\mu_{\ol{\tau}} U_{\ol{\tau}}' + U_{\ol{\tau}} &= \integ{y}{0}[\infty]<K(z-y)S_{\theta,\ol{\tau}}(U_{\ol{\tau}}(y))>, \\
v_\theta V_\theta ' + V_\theta &= \integ{y}{0}[\infty]<K(z-y)S_{\theta,0}(V_\theta(y))> ,
\end{align*}
we can use the uniform convergences above to establish $\norm{U_{\ol{\tau}}'-V_\theta '}_\infty \longrightarrow 0$ and $\norm{U_{\ol{\tau}}''-V_\theta ''}_\infty \longrightarrow 0$. The claim now follows from \cref{lemma: Implicit_smooth}.
\end{proof}
\begin{lemma}\label{lemma: Solution_Smooth}
For $\tau>0,$ suppose $(U_{\tau},\mu_{\tau},S_{\theta,\tau})$ is a solution satisfying (A2). Then $U_{\tau} \in C^\infty(\mathbb{R})$ with derivatives of all orders vanishing at infinity.
\end{lemma}
\begin{proof}
We prove the claim with an induction argument. From $F\equiv 0,$ we know
$$
\mu_{\tau}U_{\tau}'(z)=-U_{\tau}(z) + \integ{y}{\mathbb{R}}[]<K(y)S_{\theta,\tau}(U_{\tau}(z-y))>.
$$
Since $S_{\theta,\tau}\leq 1$ and $K\in L^1(\mathbb{R}),$ by dominated convergence theorem, it follows that $U_{\tau}'(\pm \infty)=0.$ Differentiating the right hand side, $-U_{\tau}'$ obviously exists and since it is bounded, the difference quotient applied to the integrand is dominated by $\max S_{\theta,\tau}'(\cdot) \norm{U_{\tau}'}_\infty |K(y)|,$ which is integrable. Hence
$$
\mu_{\tau}U_{\tau}''(z)=-U_{\tau}'(z) + \integ{y}{\mathbb{R}}[]<K(y)S_{\theta,\tau}'(U_{\tau}(z-y))U_{\tau}'(z-y)>,
$$ 
and again, by dominated convergence theorem, $U_{\tau}''(\pm \infty)=0.$ Since $S_{\theta,\tau}$ is smooth and all derivatives have compact support, this process can keep repeating, proving the claim.
\end{proof}
Combining \cref{lemma: Solution_Smooth,lemma: Unique,lemma: Implicit_smooth}, the proof of \cref{thm: perturb} is complete.

\subsection{Continuation Process} \label{subsec: Continuation Process}
Up to this point, we have shown that at a given solution, the implicit function theorem can be applied. Our base point for the first perturbation is the Heaviside solution $(V_\theta,v_\theta).$ However, after the first time, the implicit function theorem only guarantees the existence of a neighborhood $[\tau,\ol{\tau})$ where solutions persist. As we repeat this process, we need a method for passing through limits, ensuring that the continuation never gets stuck.

For example, consider the following continuation path. Starting with $\tau=0,$ there exists $\epsilon_1>0$ such that (unique) solutions $(U_{\tau},\mu_{\tau})$ satisfying (A2) exist for $\tau\in[0,\epsilon_1]$ with $\Delta_{\tau}(\xi)< \sigma(\theta).$ Starting from base solution $(U_{\epsilon_1},\mu_{\epsilon_1}),$ there exists $\epsilon_2>0$ such that solutions satisfying (A2) exist for $\tau\in[\epsilon_1,\epsilon_1+\epsilon_2]$ (and therefore, $[0,\epsilon_1+\epsilon_2]$) with $\Delta_{\tau}(\xi)< \sigma(\theta).$ As this process repeats countably many times, there are several possibilities, prior to the analysis in this subsection.

One possibility is that the process proceeds until for some finite step in the iteration, $\Delta_{\tau}(\tau)=\sigma(\theta)$ is reached and we are done; solutions certainly may still exist beyond this point, but we have not proven that they do. The other is that $\sum_{j=1}^\infty \epsilon_j = \ol{\tau}\leq 1-2\theta$ and for all $\tau\in[0,\ol{\tau}),$ solutions exist satisfying (A2) with $\Delta_{\tau}(\xi)\leq \sigma(\theta).$ If we can show a solution $(U_{\ol{\tau}},\mu_{\ol{\tau}})$ exists satisfying (A2) with $\Delta_{\ol{\tau}}(\xi)\leq \sigma(\theta),$ then we can start the process over, starting from $\ol{\tau}.$ Effectively, existence cannot break down unless our hypotheses break down.

In this subsection, we use the exact definitions of $\sigma_1(\theta)$ and $\sigma_2(\theta)$ to complete the proof of the second part of \cref{thm: continuation}, where we recall 
\theoremstyle{theorem}
\newtheorem*{thm_cont}{Theorem \ref{thm: continuation}}
\begin{thm_cont}[Continuation Criteria]
Suppose at $\tau\in(0,1-2\theta),$ a solution $(U_{\tau},\mu_{\tau})$ exists satisfying assumption (A2) with $\Delta_{\tau}(\xi)\leq \sigma(\theta).$ Then there exists $\ol{\tau}>\tau$ such that for all $\delta\in[\tau,\ol{\tau}),$ a solution $(U_{\delta},\mu_{\delta})$ exists. Moreover, if for all $\delta\in[\tau,\ol{\tau}),$ solutions satisfy (A2) with $\Delta_{\delta}(\xi)\leq \sigma(\theta),$  then there exists a solution $(U_{\ol{\tau}},\mu_{\ol{\tau}})$ satisfying (A2) with $\Delta_{\ol{\tau}}(\xi)\leq \sigma(\theta).$
\end{thm_cont}
The technique breaks down into two main parts. The first part, also seen in \cite{ExistenceandUniqueness-ErmMcLeod, Bates1997}, is to show the corresponding solution sets 
\begin{align*}
\mathcal{U}:=\{U_{\delta}: \delta \in [\tau,\ol{\tau}) \}, \qquad \mathcal{M}:=\{\mu_{\delta}:  \delta \in [\tau,\ol{\tau})  \},
\end{align*}
satisfy the requirements to apply the Arzel\'{a}--Ascoli and Bolzano--Weierstrass theorems respectively. Therefore, for a subsequence $\{\delta_n\}$ with $\delta_n \to \ol{\tau},$ the limiting solution $(U_{\delta_n},\mu_{\delta_n}) \to (U_{\ol{\tau}},\mu_{\ol{\tau}})$ exists. The second part is to show $(U_{\ol{\tau}},\mu_{\ol{\tau}})$ satisfies (A2) with $\Delta_{\ol{\tau}}(\xi)\leq \sigma(\theta).$ In all cases, we fix the translation so that $U_{\delta}^{-1}(\theta)=0$ for all $\delta.$

We note that we have only proved uniqueness for sufficiently small $\tau.$ Since we have not proven uniqueness for larger $\tau$ (although we suspect it holds true), for a given $\delta,$ there may be multiple solution choices for $(U_{\delta},\mu_{\delta})$. Therefore, the set $\mathcal{U}\times \mathcal{M}$ is understood to consist of one arbitrary pair for each $\delta.$ In a similar manner, $(U_{\ol{\tau}},\mu_{\ol{\tau}})$ may not be unique, even though it is a unique limit of some subsequence.

The following lemma shows that the Arzel\'{a}--Ascoli theorem can be applied to $\mathcal{U}.$ 

\begin{lemma}\label{lemma: U_uniformly_bounded}
For the set $\mathcal{U},$ the following properties hold.
\begin{itemize}
\item[(i)] $\mathcal{U}$ is bounded in the norm $\norm{\cdot}_{2,\infty}.$
\item[(ii)] $\mathcal{U}$ is uniformly equicontinuous.
\item[(iii)] For all $U_{\delta} \in \mathcal{U},$ the limits
$$
U_{\delta}(-\infty)=0, \qquad U_{\delta}(\infty)=1, \qquad U_{\delta}^{(k)}(\pm \infty)=0, \qquad \textnormal{for } k=1,2,
$$
hold uniformly in $\delta.$
\end{itemize}
\end{lemma}
\begin{proof}
\begin{itemize}
\item[(i)] For all $\delta$, write the solutions in the form
\begin{align} \label{eq: U_bound_Arz}
|U_{\delta}^{(j)}(z)|=\left|\int_{-\infty}^0 e^x \int_\mathbb{R} K^{(j)}(z+\mu_{\delta} x -y)S_{\theta,\delta}(U_{\delta}(y))\,\mathrm{d}y\mathrm{d}x \right| \leq \norm{K^{(j)}}_1
\end{align}
for $j=0,\, 1.$ Since $K$ has a weak first order derivative, we may write the $j=1$ case as
\begin{align*}
|U_{\delta}'(z)|&=\left|\int_{-\infty}^0 e^x \int_\mathbb{R} K(z+\mu_{\delta} x -y)S_{\theta,\delta}'(U_{\delta}(y))U_{\delta}'(y)\,\mathrm{d}y\mathrm{d}x \right|
\intertext{so}
|U_{\delta}''(z)|&=\left|\int_{-\infty}^0 e^x \int_\mathbb{R} K'(z+\mu_{\delta} x -y)S_{\theta,\delta}'(U_{\delta}(y))U_{\delta}'(y)\,\mathrm{d}y\mathrm{d}x \right| \\
&\leq \left[\max_{\delta \in [\tau,\ol{\tau}]}S_{\theta,\delta}'(\cdot)\right]\norm{U_{\delta}'}_\infty \norm{K'}_1 \leq \left[\max_{\delta \in [\tau,\ol{\tau}]}S_{\theta,\delta}'(\cdot)\right] \norm{K'}_1 ^2
\end{align*}
by \eqref{eq: U_bound_Arz}.
\item[(ii)] Let $z^*$ be fixed. Then for $z$ near $z^*,$
\begin{align*}
|U_{\delta}(z)-U_{\delta}(z^*)|&=\left|\int_{-\infty}^0 e^x \int_\mathbb{R} K(\mu_{\delta}x + y)\left[S_{\theta,\delta}(U_{\delta}(z-y))-S_{\theta,\delta}(U_{\delta}(z^*-y))\right]\,\mathrm{d}y\mathrm{d}x\right| \\
&\leq |z-z^*|\left[\max_{\delta \in [\tau,\ol{\tau}]}S_{\theta,\delta}'(\cdot)\right]\norm{U_{\delta}'}_\infty \norm{K}_1 \\
&\leq |z-z^*|\left[\max_{\delta \in [\tau,\ol{\tau}]}S_{\theta,\delta}'(\cdot)\right]\norm{K'}_1 \norm{K}_1.
\end{align*}
Therefore, $\mathcal{U}$ is uniformly equicontinuous. A similar argument applies to the derivatives.
\item[(iii)] We write $U_{\delta}$ in the form
\begin{align*}
U_{\delta}(z)&=\int_{-\infty}^0 e^x \int_{-\infty}^\infty K(y) S_{\theta,\delta}(U_{\delta}(z+\mu_{\delta}x-y))\, \mathrm{d}y\mathrm{d}x \\
&=\int_{-\infty}^0 e^x \int_{-\infty}^{z+\mu_{\delta}x} K(y) S_{\theta,\delta}(U_{\delta}(z+\mu_{\delta}x-y))\, \mathrm{d}y\mathrm{d}x
\end{align*}
As $z \to -\infty,$ the result is trivial since
$$
|U_{\delta}(z)|\leq \int_{-\infty}^0 e^x\int_{-\infty}^z |K(y)|\,\mathrm{d}y\mathrm{d}x=\integ{y}{-\infty}[z]<|K(y)|> \longrightarrow 0
$$
uniformly. For the limits as $z \to \infty$, we notice that outside of large bounded intervals $\mathcal{I}_y, \, \mathcal{I}_x,$ the integrals with respect to $y$ and $x$ are as small as we desire, uniform in $(U_{\delta},\mu_{\delta}),$ since $S_{\theta,\delta} \leq 1$ and $0\leq \mu_{\delta} \leq v_\theta.$ Now let $\mathcal{I}_y, \, \mathcal{I}_x$ be fixed.  For each $\delta,$ we have $S_{\theta,\delta}(U_{\delta}(\eta))=1$ for $\eta \geq U_{\delta}^{-1}(\theta+\delta).$ But we also have the bound $U_{\delta}^{-1}(\theta+\delta)\leq \sigma(\theta).$ Therefore, there exists a large number $T,$ uniform in $\delta,$ such that for all $z\geq T$,
$$
U_{\delta}(z) \approx \int_{\mathcal{I}_x} e^x \int_{\mathcal{I}_y} K(y) \, \mathrm{d}y\mathrm{d}x \approx 1,
$$
with the error being uniform in $\delta.$ Therefore, $U_{\delta}(\infty)=1$ uniformly. The uniform limits $U_{\delta}^{(j)}(\pm \infty)=0$ for $j=1,\, 2$ follow by the same argument but replacing $K$ with its derivatives.
\end{itemize}
\end{proof}
The previous lemma shows that the Arzel\'{a}--Ascoli theorem can be applied to $\mathcal{U}.$ Moreover, $\mathcal{M}$ is obviously bounded below by zero and above by $v_\theta$ so the Bolzano--Weierstrass theorem can be applied.

Combined, there exists a subsequence $
\{\delta_n\}$ with $\delta_n \to \ol{\tau}$ such that 
$$
(U_{\delta_n},\mu_{\delta_n},S_{\theta,\delta_n}) \to (U_{\ol{\tau}},\mu_{\ol{\tau}},S_{\theta,\ol{\tau}})
$$
exists with respect to the norm $\norm{\cdot}_{2,\infty}+|\cdot|+\norm{\cdot}_1$. By applying the dominated convergence theorem, we can show that $(U_{\ol{\tau}},\mu_{\ol{\tau}},S_{\theta,\ol{\tau}})$ is a solution in that it solves $F[U_{\ol{\tau}},\mu_{\ol{\tau}},S_{\theta,\ol{\tau}}](z)=0.$ Finally, we complete the proof of \cref{thm: continuation} by showing (A2) is satisfied and $\Delta_{\ol{\tau}}(\xi)\leq \sigma(\theta)$ for the limiting solution.

At the very least, the desired limits $U_{\ol{\tau}}(-\infty)=0,$ $U_{\ol{\tau}}(\infty)=1$ hold by \cref{lemma: U_uniformly_bounded} (iii) and $U_{\ol{\tau}}'\geq 0$ when $U_{\ol{\tau}} \in [\theta,\theta+\ol{\tau}].$ Define
\begin{align*}
\alpha &:= \max\{z: U_{\ol{\tau}}(z)=\theta\}, \\
\beta &:= \min\{z: U_{\ol{\tau}}(z)=\theta+\ol{\tau}\}.
\end{align*}
We may translate so that $\alpha=0.$ Note that $U_{\ol{\tau}}'\geq (\not\equiv)\, 0$ on $[0,\beta]$ and $U_{\ol{\tau}}'=0$ for any points where $U_{\ol{\tau}} \in \{\theta,\theta+\ol{\tau} \}$ outside of $[0,\beta].$ Certainly $\beta \leq \sigma(\theta)=\min\{\sigma_1(\theta),\sigma_2(\theta)\}$ since the claim holds for all $\delta.$  It suffices to show that $U_{\ol{\tau}}'(z)>0$ on $[0,\beta],$ $U_{\ol{\tau}}(z)<\theta$ on $(-\infty,0),$ and $U_{\ol{\tau}}(z)>\theta+\ol{\tau}$ on $(\beta,\infty).$ Therefore, assumption (A2) is satisfied.

The most challenging task is verifying the threshold conditions outside of $[0,\beta].$ In general, the difficulty is that we require information about the wave shape when the closed form is not available. In particular, the region
$$
(-M,-M+\beta)\cup (M,M+\beta) \subset [0,\beta]^c
$$
is the most difficult because it is in this region where precise information about $U_{\ol{\tau}}$ is required in order to determine if there are any critical points. This is where the estimate $\beta\leq \sigma_2(\theta)$ provides value since this bound guarantees  $U_{\ol{\tau}} \not\in \{\theta,\theta+\ol{\tau} \}$ inside this region. We prove this fact in the following lemma.
\begin{lemma}\label{lemma: alphbeta_critpts}
\begin{itemize}
\item[(i)] Suppose $\beta \leq \sigma_2(\theta)<M$ and $U_{\ol{\tau}}$ has a critical point at $z_* \in (-M,-M+\beta).$ Then $U_{\ol{\tau}}(z_*)<\theta.$
\item[(ii)] Suppose $\beta \leq \sigma_2(\theta)<M$ and $U_{\ol{\tau}}$ has a critical point at $z^* \in (M,M+\beta).$ Then $U_{\ol{\tau}}(z^*)>\theta+\ol{\tau}.$
\end{itemize}
\end{lemma}
\begin{proof}
\begin{itemize}
\item[(i)] From the equation
$$
\mu_{\ol{\tau}}U_{\ol{\tau}}' + U_{\ol{\tau}} = \integ{y}{\mathbb{R}}[]<K(z-y)S_{\theta,\ol{\tau}}(U_{\ol{\tau}}(y))>,
$$
if $U_{\ol{\tau}}'(z_*)=0$ for some $z_*,$ then
\begin{align*}
U_{\ol{\tau}}(z_*)&=\integ{y}{0}[\infty]<K(z_*-y)S_{\theta,\ol{\tau}}(U_{\ol{\tau}}(y))> \\
&= \integ{y}{-\infty}[z_*-\beta]<K(y)> + \integ{y}{z_*-\beta}[z_*]<K(y)S_{\theta,\ol{\tau}}(U_{\ol{\tau}}(z_*-y))> \\
&<  \integ{y}{-\infty}[-M-\beta]<K(y)> +S_{\theta,\ol{\tau}}(U_{\ol{\tau}}(z_*+M)) \integ{y}{z_*-\beta}[z_*]<K(y)>\\
&< \integ{y}{-\infty}[-M-\beta]<K(y)> + \integ{y}{-M}[-M+\beta]<K(y)> \\
&< \integ{y}{-\infty}[-M-\sigma_2(\theta)]<K(y)> + \integ{y}{-M}[-M+\sigma_2(\theta)]<K(y)>=\theta
\end{align*}
by the definition of $\sigma_2.$
\item[(ii)] Similar to (i), if $U_{\ol{\tau}}'(z^*)=0$ for some $z^*$, then
\begin{align*}
U_{\ol{\tau}}(z^*)&> \integ{y}{-\infty}[M-\beta]<K(y)> + S_{\theta,\ol{\tau}}(U_{\ol{\tau}}(z^*-M)) \integ{y}{z^*-\beta}[z^*]<K(y)> \\
&> \integ{y}{-\infty}[M-\beta]<K(y)> + \integ{y}{M}[M+\beta]<K(y)> \\
&=1-\left(\integ{y}{M-\beta}[M]<K(y)>+\integ{y}{M+\beta}[\infty]<K(y)>\right).\\
\intertext{By symmetry, the previous term can be written as}
&\phantom{=}1-\left(\integ{y}{-M}[-M+\beta]<K(y)>+\integ{y}{-\infty}[-M-\beta]<K(y)>\right)\\
&> 1-\left(\integ{y}{-M}[-M+\sigma_2(\theta)]<K(y)>+\integ{y}{-\infty}[-M-\sigma_2(\theta)]<K(y)>\right) \\
&=1-\theta.
\end{align*}
Recalling $1-2\theta\geq\ol{\tau},$ it follows that $1-\theta\geq\theta+\ol{\tau}$ and the claim follows.
\end{itemize}
\end{proof}
The remaining region we have not dealt with yet is 
$$
(-\infty,-M]\cup [-M+\beta,M]\cup[M+\beta,\infty),
$$
but it is easy to see the possible behavior of $U_{\ol{\tau}}'$ on these intervals. Define 
$$
 h(z):=U_{\ol{\tau}}'(z) e^{\f{z}{\mu_{\ol{\tau}}}},
$$
which notably, has the same sign as $U_{\ol{\tau}}'.$ The following lemma will account for the behavior of $U_{\ol{\tau}}'$ on all three intervals.
\begin{lemma}\label{lemma: h_inc_dec}
\begin{itemize}
\item[(i)] The function $h$ is strictly decreasing on $(-\infty,-M]\cup[M+\beta,\infty).$
\item[(ii)] The function $h$ is strictly increasing on $[-M+\beta,M].$
\end{itemize}
\end{lemma}
\begin{proof}
A simple calculation shows
$$
h'(z)= \f{e^{\f{z}{\mu_{\ol{\tau}}}}}{\mu_{\ol{\tau}}} \integ{y}{0}[\beta]<K(z-y)S_{\theta,\ol{\tau}}'(U_{\ol{\tau}}(y))U_{\ol{\tau}}'(y) >.
$$
Recalling that $K(\cdot)<0$ on $(-\infty,-M)\cup (M,\infty)$ and $K(\cdot)>0$ on $(-M,M),$ it can easily be seen that the integrand is negative for the regions described in (i) and positive for the one in (ii).
\end{proof}
We are finally ready to complete the proof of \cref{thm: continuation} with one final lemma.

\begin{lemma}
The limiting solution $U_{\ol{\tau}}$ satisfies assumption (A2) with $\Delta_{\ol{\tau}}(\xi)\leq \sigma(\theta).$
\end{lemma}
\begin{proof}
We track the solution on $\mathbb{R}$ and show that all threshold requirements are met.

On $(-\infty,-M]:$ Starting with $U_{\ol{\tau}}(-\infty)=0,$ by \cref{lemma: h_inc_dec} (i),  $U_{\ol{\tau}}$ decreases.  

On $(-M,-M+\beta):$ As discussed, the only way $U_{\ol{\tau}}=\theta$ can occur is at a local maximum. But by \cref{lemma: alphbeta_critpts} (i), any possible local maximums that occur on this interval stay below $\theta.$ 

On $[-M + \beta,M]:$ The function $h$ is increasing by \cref{lemma: h_inc_dec} (ii) and since $U_{\ol{\tau}}' \geq 0$ on $[0,\beta]\subset [0,M],$ there are only two possibilities. If $h(-M+\beta)<0,$ then $h$ changes signs once, from negative to positive; therefore, $U_{\ol{\tau}}$ has exactly one critical point, a local minimum, for some $z_*\in (-M+\beta,0)$. On the other hand, if $h(-M+\beta)\geq 0,$ then $U_{\ol{\tau}}$ is strictly increasing. In either case, we may conclude that $U_{\ol{\tau}}' > 0$ on $[0,M]\supset [0,\beta]$ and $U_{\ol{\tau}}$ crosses each threshold exactly once.

On $(M,M+\beta):$ Any possible local minimums stay above $\theta+\ol{\tau}$ by \cref{lemma: alphbeta_critpts} (ii). 

On $[M+\beta,\infty):$ The function $h$ is decreasing by \cref{lemma: h_inc_dec} (i), leaving only three possibilities. If  $h(M+\beta)<0,$ then we must have $U_{\ol{\tau}}'<0$. If $h(M+\beta)\geq 0$ and $h(\infty)\geq 0,$ then $U_{\ol{\tau}}'>0$. Finally, if $h(M+\beta)\geq 0$ and $h(\infty) <0,$ $h$ must change signs exactly once, from positive to negative. Therefore, $U_{\ol{\tau}}$ has exactly one critical point, a local maximum, in which $U_{\ol{\tau}}'<0$ thereafter. In all three cases, we see that $U_{\ol{\tau}}(\infty)=1$ and $U_{\ol{\tau}}>\theta+\ol{\tau}$ on this interval.

In conclusion, $U_{\ol{\tau}}$ satisfies assumption (A2) with $\Delta_{\ol{\tau}}(\ol{\tau})=\beta\leq \sigma(\theta).$
\end{proof}
Combining all lemmas in this subsection, we have completed the proof of \cref{thm: continuation}. The work in this section rigorously established the existence of fronts under continuous changes in $\tau.$ 

\section{Example of Existence for all $\tau$} \label{sec: Example}
In \cref{thm: cont_taustar}, we developed an {\it a priori} existence result by establishing a number $\tau^*(\theta)>0$ such that existence holds for $\tau \leq \tau^*(\theta).$ The number $\tau^*(\theta)$ is obtained by calculating $\sigma(\theta),$ plugging it into the function $\Phi(\theta,\tau)=\tau-\omega_L(\sigma(\theta),\tau),$ and scanning for the point where the increasing (in $\tau$) function $\Phi$ crosses zero. All of these calculations are easy to perform numerically.

The purpose of this section is to show that our requirement $\Delta_{\tau}(\xi)\leq \sigma(\theta)$ can be quite nonrestrictive. We perform a numerical example where continuation may proceed from $\tau=0$ to $\tau=1-2\theta,$ where a standing front exists.
\\ \\
\noindent{\bf Numerical Method.} All solutions are obtained by approximating sigmoidal firing rates by $N$ Heaviside step functions and using a numerical root solver to find discrete points $0=\Delta_0<\Delta_1<...<\Delta_N$ such that $U_{\tau}(\Delta_k)=\theta+\f{k}{N}\tau.$ The wave speed is solved from $U_{\tau}(0)=\theta.$ In total, there are $N+1$ equations and $N+1$ unknowns. In the figures below, $N=50$ and the right end point method is used to approximate $S_{\theta,\tau}.$

 \subsubsection*{Firing Rate}
 Inspired by the work in \cite{CoombesSchmidt}, we define
 \begin{equation}
 \label{eq: smooth_heav_exp}
 S_{\theta,\tau}(u) = 
 \begin{cases}
 0 & u\leq \theta, \\
 \integ{x}{0}[u-\theta]<A(\tau)\e{\f{r}{x(x-\tau)}}> & \theta<u<\theta+\tau, \\
 1 & u\geq \theta + \tau.
 \end{cases}
 \end{equation}
Choose $r=0.01$. The function $A(\tau):=\frac{1}{\int_0^\tau \exp\left(\frac{r}{x(x-\tau)}\right)\, \mathrm{d}x}$ is a normalizing constant. Note that $S_{\theta,\tau}$ is odd symmetric about its inflection point.

\subsubsection*{Kernel}
Keeping with standard examples from the literature, we choose
\begin{equation}\label{eq: ex_kmh}
K(x)=Ae^{-a|x|}-Be^{-b|x|},
\end{equation}
with $A=5, a=0.5, B=4$. From the assumption $\int_\mathbb{R} K=1$, we find $b=\frac{2aB}{2A-a}=0.4211$. Note that $M=\frac{\ln(A/B)}{a-b}=2.8282$.
\subsubsection*{Parameter Calculations}
Fix $\theta=0.2.$ Recall the following definitions. The parameter $\sigma_1(\theta)\in (0,M)$ is the unique constant such that $V_\theta'(-\sigma_1(\theta))=0,$ or equivalently,
$$
\integ{x}{-\infty}[0]<e^{\f{x}{v_\theta}}K(x-\sigma_1(\theta))>=0,
$$
where
$$
\phi(v_\theta)=\integ{x}{-\infty}[0]<e^{\f{x}{v_\theta}}K(x)>=\f{2}-\theta.
$$
Define $\sigma_2(\theta)$ to be the positive constant that is the unique solution to the equation
\begin{align}
\left(\int_{-\infty}^{-M-\sigma_2(\theta)}+\int_{-M}^{-M+\sigma_2(\theta)}\right) K(x) \, \mathrm{d}x &=\theta. \tag{\ref{eq: sigma2}}
\end{align}
Numerical calculations suggest $\sigma_1(\theta)=2.4860$ and $\sigma_2(\theta)=2.2953$ so $\sigma(\theta)=\min\{\sigma_1(\theta),\,\sigma_2(\theta)\}=2.2953$.
Using these parameters, recall the increasing function 
\begin{equation}\tag{\ref{eq: Phi_tau}}
\Phi(\theta,\tau):=\tau-\omega_L(\sigma(\theta),\tau).
\end{equation}
where
\begin{align*}
\omega_L(\sigma(\theta),\tau)&=
\begin{cases}
\displaystyle
\min_{\substack{y \in [0,\sigma(\theta)] \\ \delta\in [\theta,\theta+\tau]}} (V_\delta(\sigma(\theta)-y)-V_\delta(-y)), &\tau<\f{2}-\theta, \\
\displaystyle\min_{\substack{y \in [0,\sigma(\theta)] \\ \delta\in [\theta,\f{2}]}} (V_\delta(\sigma(\theta)-y)-V_\delta(-y)), &\tau\geq\f{2}-\theta,
\end{cases} \\
V_\delta(z)&=\f{v_\delta} \int_0^\infty \int_{-\infty}^0 e^{\f{x}{v_\delta}}K(x+z-y)\,\mathrm{d}x\mathrm{d}y \\
&=\integ{x}{-\infty}[z]<K(x)>-\integ{x}{-\infty}[z]<e^{\f{x-z}{v_\delta}}K(x)>, 
\end{align*}
and $v_\delta$ uniquely solves
$$
\phi(v_\delta)=\integ{x}{-\infty}[0]<e^{\f{x}{v_\delta}}K(x)>=\f{2}-\delta
$$
for $\delta<\f{2}.$ 
\subsubsection*{Results}
According to \cref{thm: cont_taustar}, if $\Phi(\theta,1-2\theta)\leq 0,$ then $\tau^*(\theta)=1-2\theta$ and traveling fronts exist for $\tau\in[0,1-2\theta)$ with a standing front existing when $\tau=1-2\theta.$ Here, $1-2\theta=0.6.$ Moreover, all solutions satisfy $\Delta_{\tau}(\xi)\leq \sigma(\theta).$ This is indeed the case; in Figures \ref{fig: Phi_plot_ex2} and \ref{fig: Front_ex2_and_Front_zoomed_ex2} below, we highlight these descriptions and plot the standing front.
\begin{figure}[H]
\centering
\subfigure[]
{\includegraphics[width=70mm]{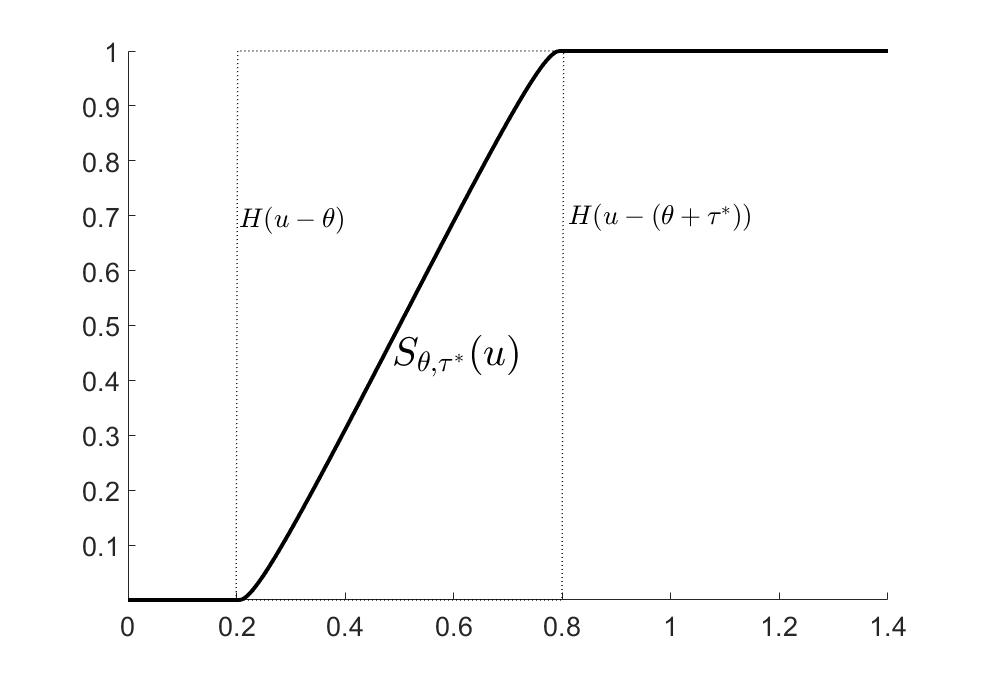}}
\subfigure[]
{\includegraphics[width=70mm]{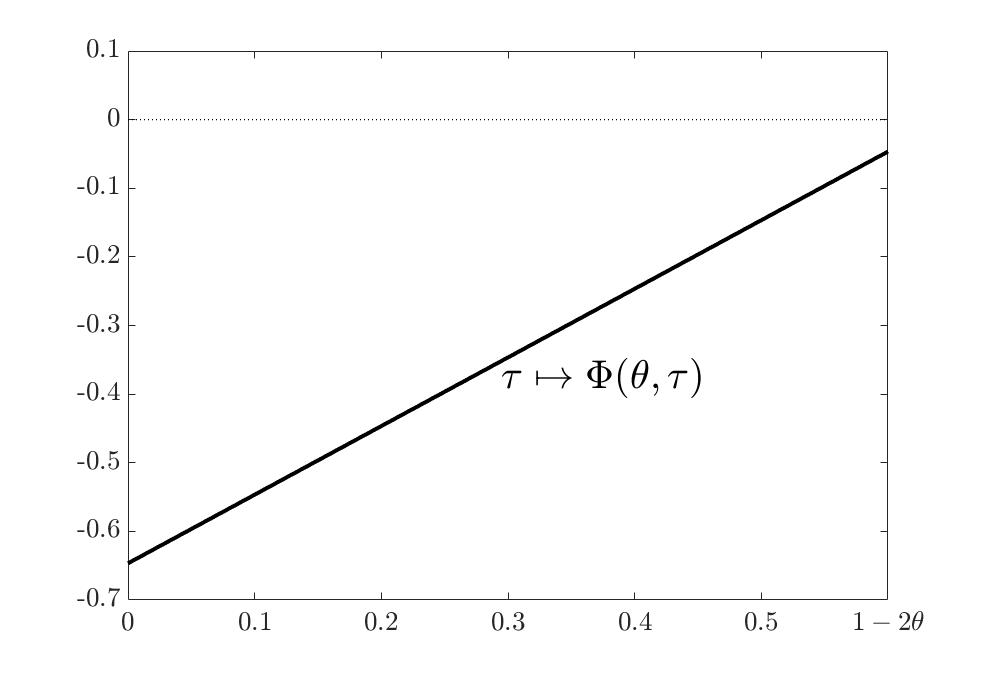}}
\caption{(a) Plot of $S_{\theta,\tau^*}(u).$ The dotted vertical lines denote $H(u-\theta)$ and $H(u-(\theta+\tau^*(\theta)))$ respectively. (b) Plot of $\tau \mapsto \Phi(\theta,\tau).$ We are guaranteed $\Delta_{\tau}(\tau)<\sigma(\theta)$ when $\tau\leq 1-2\theta.$ Hence, our hypotheses never fail.}
\label{fig: Phi_plot_ex2}
\end{figure}
\begin{figure}[H] 
\centering
\subfigure[]{\includegraphics[width=70mm]{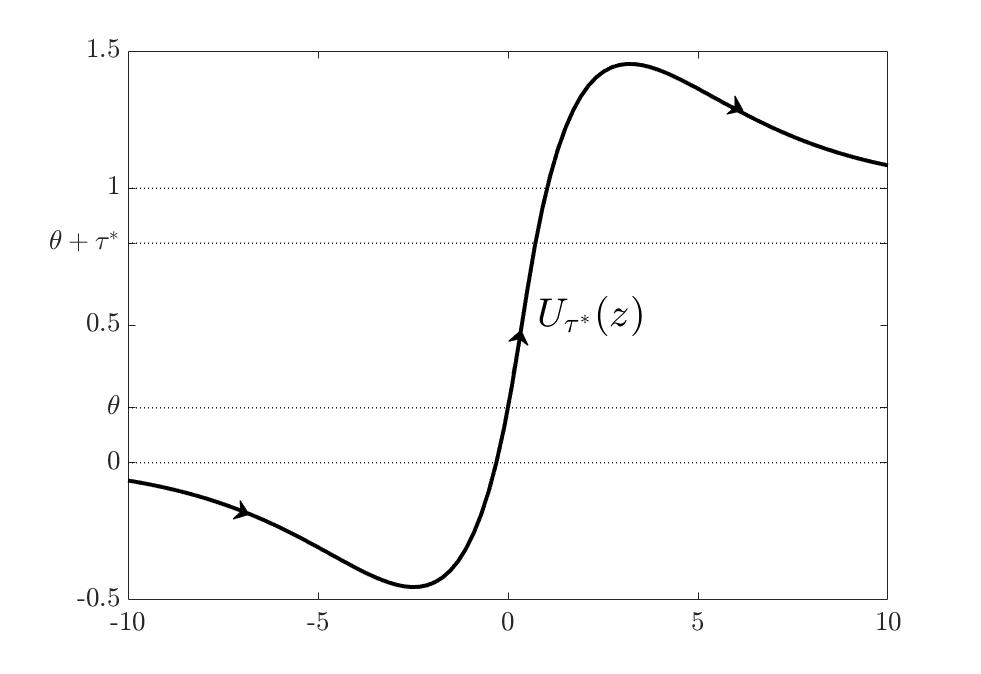}}
\subfigure[]{\includegraphics[width=70mm]{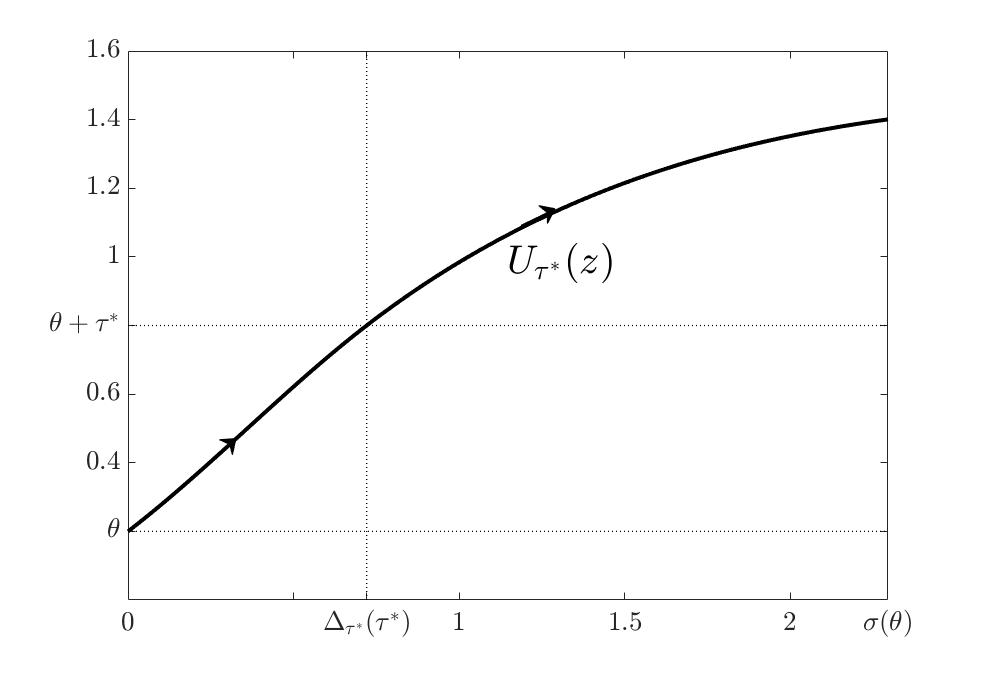}}
\caption{(a) Plot of $U_{\tau^*}$ when $\mu_{\tau^*}=0.$ Note that $U_{\tau^*}$ is increasing through the threshold region. (b) Zoomed in plot of $U_{\tau^*},$ showing $\Delta_{\tau^*}(\tau^*(\theta))<\sigma(\theta).$} 
\label{fig: Front_ex2_and_Front_zoomed_ex2}
\end{figure}
Since fronts exist for all $\tau \in[0,1-2\theta]$, we can plot the function $\tau \mapsto \mu_{\tau}$. Intuitively, since the function $\tau \mapsto S_{\theta,\tau}(u)$ is decreasing in $\tau$, we expect slower firing rates to lead to slower traveling waves. This is indeed the case. See Figure \ref{fig: wavespeed_plot_ex2_SIADS}.
\begin{figure}[H]
\centering
\includegraphics[width=75mm]{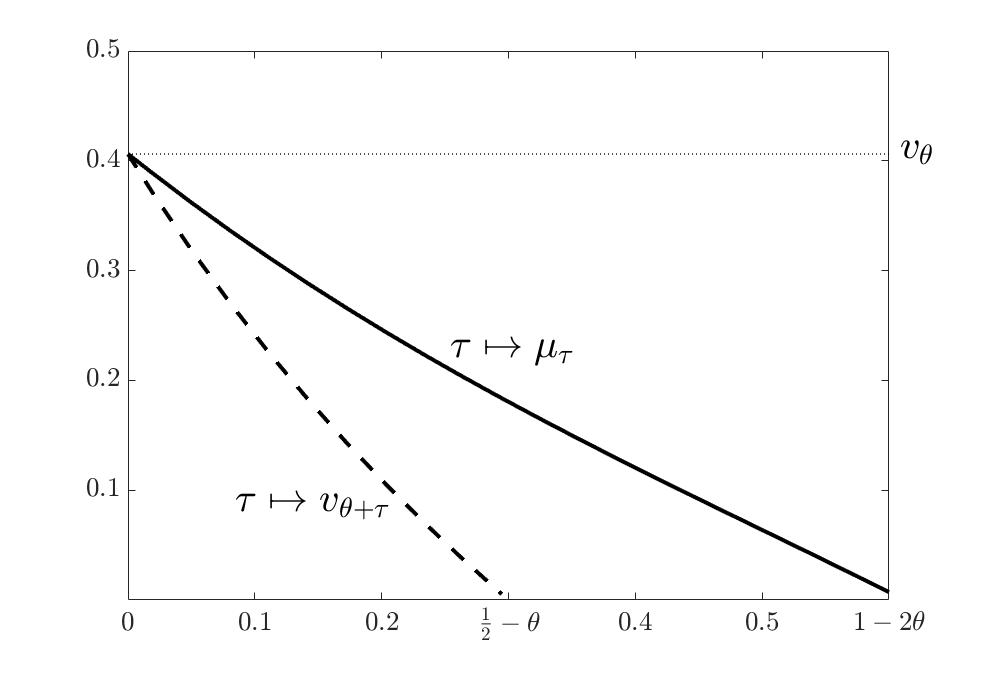}
\caption{Plot of the wave speed as a function of $\tau$. The functions $\tau \mapsto \mu_{\tau}$ and $\tau \mapsto v_{\theta+\tau}$ are decreasing. Moreover, $v_{\theta+\tau}\leq \mu_{\tau}<v_\theta$ holds, as predicted by \cref{lemma: Global_Wave_Bounds}.}
\label{fig: wavespeed_plot_ex2_SIADS}
\end{figure}
\section*{Discussion}
In the present study, we applied the powerful homotopy technique in \cite{ExistenceandUniqueness-ErmMcLeod} in order to prove the existence of traveling fronts in neural field models with lateral inhibition kernels and smooth Heaviside firing rates. Our results expand the existence and uniqueness results in \cite{Zhang-HowDo} by exploring the traveling wave problem beyond models with Heaviside firing rates, which are less realistic biologically. To the author's knowledge, beyond the landmark study of monotone fronts in \cite{ExistenceandUniqueness-ErmMcLeod}, our problem was previously unsolved in a rigorous setting. Unlike in \cite{ExistenceandUniqueness-ErmMcLeod}, we had to carefully handle the fact that kernels with inhibition add great difficulty when repeatedly applying the implicit function theorem over Banach spaces.

Given sigmoidal firing rates $S_{\theta,\tau}$, the greatest advancement of this work is coming up with {\it a priori} lower bounds $\tau^*(\theta)$ such that existence holds for $\tau\leq \tau^*(\theta).$ This bound is obtained based on global comparisons between waves arising from sigmoidal versus Heaviside firing rates. In \cref{sec: Example}, we worked through a reasonable example that indicated that many traveling waves are proven to exist based on $\tau^*(\theta).$

However, our study leads to a variety of open problems. For one, we do not have a method to prove uniqueness beyond the $\tau \ll 1$ case. In part, this is because our uniqueness proof relies on proving that solutions can be trapped between Heaviside solutions. When $\tau \ll 1$ does not hold, the bounds are not good enough to prove the result. We suspect uniqueness holds since it holds for all thresholds and lateral inhibition kernels in the Heaviside case \cite{Dyson2019_MBE,Zhang-HowDo}.

A related topic of importance that we did not study is stability. In the Heaviside case, the so-called Evan's function method is common, but not proven to be available to us. Without an Evan's function to classify the spectrum from the linearization, we may need a way of comparing---or even constructing---sub and super solutions, as in \cite{ExistenceUniqueness-Chen}. This task seems difficult since $K<0$ occurring disrupts some of the obvious behaviors of the model when $K>0$ only.

Finally, we realize that the front is typically most valuable when viewed as the fast $O(\epsilon)$-time jump in the pulse solution to the singularly perturbed system \eqref{eq: intro_system_1}-\eqref{eq: intro_system_2}. Although the analysis is highly nontrivial, this work leads us to believe that for most kernel choices, traveling pulses also exist when $\epsilon \ll 1.$ A novel adaptation of the compelling result in \cite{Faye2015} may help us achieve the result.

\section*{Appendix A} \label{sec: app_cont}
\subsection*{Application of the Implicit Function Theorem}
{\bf Proof of \cref{lemma: Implicit_smooth}. } 
\begin{proof}[\unskip\nopunct] 
Using standard methods like those in \cite{accinelli2010generalization}, we first show that $\epsilon_1$ and $\delta_1$ may be chosen so that for fixed $S_{\theta,\ol{\tau}}$ with $\norm{S_{\theta,\ol{\tau}}-S_{\theta,\tau}}_1<\delta_1,$ the function $\mathcal{N}$ is a contraction mapping. Then we show we may choose $\epsilon_0$ so that $\mathcal{N}$ maps the ball $E(U_{\tau},\mu_{\tau},\ol{\psi}_{\tau})\times \mathbb{R} \cap \mathcal{B}((U_{\tau},\mu_{\tau}); \epsilon_0)$ into itself. We use the following notation: $$\norm{(\ol{U},\ol{\mu})}=\norm{\ol{U}}_{2,\infty} + |\ol{\mu}|, \qquad L:=DF_{U,\mu}[U_{\tau},\mu_{\tau},S_{\theta,\tau}].
$$

For the first part, let $S_{\theta,\ol{\tau}}$, $(\ol{U}_1,\ol{\mu}_1),$ and $(\ol{U}_2,\ol{\mu}_2)$ be fixed. Then
\begin{align*}
\label{eq: contraction}
&\norm{\mathcal{N}[\ol{U}_2,\ol{\mu}_2,S_{\theta,\ol{\tau}}]-\mathcal{N}[\ol{U}_1,\ol{\mu}_1,S_{\theta,\ol{\tau}}]} \\
&=\norm{L^{-1}(L(\ol{U}_2-\ol{U}_1,\ol{\mu}_2-\ol{\mu}_1)-(F[\ol{U}_2,\ol{\mu}_2,S_{\theta,\ol{\tau}}]-F[\ol{U}_1,\ol{\mu}_1,S_{\theta,\ol{\tau}}]))} \\
&=\norm{L^{-1}(L(\ol{U}_2-\ol{U}_1,\ol{\mu}_2-\ol{\mu}_1)-DF_{U,\mu}[\ol{U}_1,\ol{\mu}_1,S_{\theta,\ol{\tau}}](\ol{U}_2-\ol{U}_1,\ol{\mu}_2-\ol{\mu}_1))} \\
&+\norm{L^{-1}}(o(\norm{\ol{U}_2-\ol{U}_1}_{2,\infty}) + o(|\ol{\mu}_2-\ol{\mu}_1|)) \\
&\leq \epsilon_1\norm{L^{-1}}(\norm{\ol{U}_2-\ol{U}_1}_{2,\infty}+|\ol{\mu}_2-\ol{\mu}_1|). \tag{A.4}
\end{align*}
By the continuity claims in \cref{lemma: Gat_cont}, choose $\delta_1>0,$ $\epsilon_0>0$ small so that $\norm{S_{\theta,\ol{\tau}}-S_{\theta,\tau}}_1<\delta_1,$ $(\norm{\ol{U}_2-\ol{U}_1}_{2,\infty}+|\ol{\mu}_2-\ol{\mu}_1|)<\epsilon_0.$ Then $\epsilon_1=\f{2 \norm{L^{-1}}}$ may be chosen independent of $(\ol{U}_1,\ol{\mu}_1),$ $(\ol{U}_2,\ol{\mu}_2),$  $\ol{\tau}$ so that $\mathcal{N}$ is a contraction mapping.

For the second part, let $S_{\theta,\ol{\tau}},$ $(\ol{U},\ol{\mu})$ be fixed. Then
\begin{align*}
&\norm{\mathcal{N}[\ol{U},\ol{\mu},S_{\theta,\ol{\tau}}]-(U_{\tau},\mu_{\tau})} \\
&=\norm{\mathcal{N}[\ol{U},\ol{\mu},S_{\theta,\ol{\tau}}]-\mathcal{N}[U_{\tau},\mu_{\tau},S_{\theta,\tau}]} \\
&\leq \norm{\mathcal{N}[\ol{U},\ol{\mu},S_{\theta,\ol{\tau}}]-\mathcal{N}[U_{\tau},\mu_{\tau},S_{\theta,\ol{\tau}}]} \tag{A.5}\label{eq: balltoself_one} \\
&+ \norm{\mathcal{N}[U_{\tau},\mu_{\tau},S_{\theta,\ol{\tau}}]-\mathcal{N}[U_{\tau},\mu_{\tau},S_{\theta,\tau}]}. \tag{A.6} \label{eq: balltoself_two}
\end{align*}
By the analysis for the contraction mapping in \eqref{eq: contraction},
\begin{align*}
\norm{\mathcal{N}[\ol{U},\ol{\mu},S_{\theta,\ol{\tau}}]-\mathcal{N}[U_{\tau},\mu_{\tau},S_{\theta,\ol{\tau}}]} &\leq \f{2}(\norm{\ol{U}-U_{\tau}}_{2,\infty}+|\ol{\mu}-\mu_{\tau}|)\tag{\ref{eq: balltoself_one}} \\ 
&\leq \f{\epsilon_0}{2}.
\end{align*}
By \cref{lemma: Gat_cont}, we may choose $\delta_2$ so that
\begin{align*}
\norm{\mathcal{N}[U_{\tau},\mu_{\tau},S_{\theta,\ol{\tau}}]-\mathcal{N}[U_{\tau},\mu_{\tau},S_{\theta,\tau}]}&=\norm{L^{-1}(F[U_{\tau},\mu_{\tau},S_{\theta,\ol{\tau}}]-F[U_{\tau},\mu_{\tau},S_{\theta,\tau}])} \tag{\ref{eq: balltoself_two}} \\
&\leq \f{\epsilon_0}{2}
\end{align*}
when $\norm{S_{\theta,\ol{\tau}}-S_{\theta,\tau}}_1 < \delta_2.$ Finally the choice $\delta_0=\min\{\delta_1,\delta_2\}$ guarantees 
$$
\norm{\mathcal{N}[\ol{U},\ol{\mu},S_{\theta,\ol{\tau}}]-(U_{\tau},\mu_{\tau})}\leq \epsilon_0.
$$
We may write $F[U(S_{\theta,\ol{\tau}}),\mu(S_{\theta,\ol{\tau}}),S_{\theta,\ol{\tau}}]=0$ for $\norm{S_{\theta,\ol{\tau}}-S_{\theta,\tau}}_1<\delta_0,$ with $(U(S_{\theta,\ol{\tau}}),\mu(S_{\theta,\ol{\tau}})) \in E(U_{\tau},\mu_{\tau},\ol{\psi}_{\tau})\times \mathbb{R} \cap \mathcal{B}((U_{\tau},\mu_{\tau}); \epsilon_0)$ continuous in the norm $\norm{\cdot}_{2,\infty} + |\cdot|$ with respect to changes in $S_{\theta,\ol{\tau}}$ in the norm $\norm{\cdot}_1$ by \cref{lemma: Gat_cont}.
\end{proof}

\subsubsection*{Competing Interests}
The author declares that they have no competing interests.
\subsubsection*{Acknowledgments}
The author would like to thank Lycoming College, and the Lehigh University College of Arts and Sciences for a generous summer research fellowship in the summer of 2018, under the advisement of Linghai Zhang. He also wants to thank the anonymous referees, as well as Daniel Conus, for giving feedback on early versions of the manuscript.

\bibliographystyle{siam}
\begingroup
\setlength{\bibsep}{0pt}
\bibliography{bibliography_cumulative}
\endgroup
\end{document}